\documentclass{article}
\usepackage[a4paper]{geometry}
\usepackage[english]{babel}
\usepackage[utf8]{inputenc}
\usepackage[T1]{fontenc}
\usepackage{amsmath}
\usepackage{amsthm}
\usepackage{amsfonts}
\usepackage{amssymb}
\usepackage{thmtools}
\usepackage{graphicx}
\usepackage{epstopdf}
\usepackage{caption}
\usepackage{subfig}
\usepackage{etoolbox}

\usepackage{RR}

\newbool{Preprint}
\setbool{Preprint}{true}

\usepackage[numbers]{natbib}
\newcommand{\textcite}{\citet}

\newtheorem{lem}{Lemma}

\newtheorem{thm}{Theorem}
\newtheorem{prop}{Proposition}
\newtheorem{coro}{Corollary}
\newcommand{\pen}{\mathop{\text{pen}}}

\newcommand{\degW}{d_W}
\newcommand{\degUpsilon}{d_\Upsilon}
\newcommand{\maxaW}{T_W}
\newcommand{\maxaUpsilon}{T_\Upsilon}

\newcommand{\KL}{\mathrm{KL}}
\newcommand{\JKL}{\mathrm{JKL}}

\author{L. Montuelle, E. Le Pennec and S. X. Cohen}

\title{Gaussian Mixture Regression model with logistic weights, a penalized
  maximum likelihood approach}

\RRNo{8281}
\RRdate{April 2013}
\RRauthor{Lucie Montuelle \thanks[select]{Select - Inria Saclay Idf / LM Orsay -
  Université Paris Sud} \and Erwan Le Pennec \thanksref{select} \and Serge X. Cohen \thanks[ipanema]{IPANEMA - CNRS /
  Synchrotron Soleil}}
\RRtitle{Gaussian Mixture Regression model with logistic weights, a penalized
  maximum likelihood approach}
\RRetitle{Gaussian Mixture Regression model with logistic weights, a penalized
  maximum likelihood approach}
\RRabstract{We wish to estimate conditional density using Gaussian Mixture
Regression model with logistic weights and means depending on the covariate.
We aim at selecting the number of components of this model as well as
the other parameters by a penalized maximum likelihood approach. We provide
a lower bound on penalty, proportional up to a logarithmic term to the
dimension of each model, that ensures an oracle inequality for our
estimator.  Our theoretical analysis is supported by some numerical experiments.
}
\RRresume{Nous souhaitons estimer une densité conditionelle à l'aide d'un
  modèle de mélange de régression gaussienne à poids logistiques et
  moyennes dépendant d'une covariable. L'objectif est de sélectionner
  le nombre de composantes dans le modèle ainsi que d'estimer les
  autres paramètres par une approche de type maximum de vraisemblance
  pénalisé. Nous proposons une borne inférieur sur la pénalité,
  proportionelle à un facteur logarithmique près, à la dimension de
  chaque modèle, qui assure l'existence d'une inégalité oracle pour
  notre estimateur. Notre analyse théorique est confirmée par des
  expériences numériques.
}
\RRkeyword{Conditional density estimation, Gaussian Mixture
  Regression, Model selection}
\RRmotcle{Estimation de densité conditionnelle, Mélange de régression
  gaussienne, Sélection de modèles}
\RRprojet{Select}
\RCSaclay 

\begin{document}
\makeRR

\maketitle
\section{Framework}

In classical Gaussian mixture models, density is modeled by
\begin{align*}
s_{K,\upsilon,\Sigma,w}(y)= \sum_{k=1}^{K} \pi_{w,k} \Phi_{\upsilon_k,\Sigma_k}(y),
\end{align*}
where $K \in \mathbb{N}^{*}$ is the number of mixture components,
$\Phi_{\upsilon,\Sigma}$ is the density of a Gaussian of mean $\upsilon$ and
covariance matrix $\Sigma$,
\begin{align*}
\Phi_{\upsilon,\Sigma}(y)=\frac{1}{\sqrt{(2\pi)^{p}|\Sigma|}} e^{\frac{-1}{2}(y-\upsilon)'\Sigma^{-1}(y-\upsilon)}  
\end{align*}
and mixture weights can always be defined from a $K$-tuple $(w_1,\ldots,w_K)$ with a logistic scheme:
\begin{align*}
  \pi_{w,k} = \frac{e^{w_k}}{\sum_{k'=1}^{K} e^{w_{k'}}}.
\end{align*}
In this article, we consider such a model in which mixture weights
as well as means can depend on a covariate.

More precisely, we observe $n$ pairs of random variables $((X_i,Y_i))_{1
  \leq i \leq n}$ where covariates $X_i$s are independent and
$Y_i$s are independent conditionally to the $X_i$s. We want to
estimate the conditional density $s_0(\cdot|x)$ with respect to the Lebesgue measure
of $Y$ given $X$. We model this conditional density by a mixture of Gaussian
regression with varying logistic weights
\begin{align*}
s_{K,\upsilon,\Sigma,w}(y | x)= \sum_{k=1}^{K} \pi_{w(x),k} \Phi_{\upsilon_k(x),\Sigma_k}(y),
\end{align*}
where  $(\upsilon_1,\ldots,\upsilon_K)$ and $(w_1,\ldots,w_K)$ are now
$K$-tuples of functions chosen, respectively, in a set $\Upsilon_K$ and $W_K$.
Our aim is then to estimate those functions $\upsilon_k$ and $w_k$, the
covariance matrices $\Sigma_k$ as well as the number of classes $K$ so
that the \emph{error} between the estimated conditional density and the true
conditional density is \emph{as small as possible}. 

The classical Gaussian mixture case has been much studied~\citep{FiniteMixture}. Nevertheless, theoretical properties of such model have
been less considered. In a Bayesian framework, asymptotic properties of
posterior distribution are obtained by \textcite{Choi}, \textcite{Genovese}, \textcite{VanderVaart} when the true density is assumed to be a Gaussian mixture.
AIC/BIC penalization scheme are often used to select a number of
cluster (see \textcite{burnham02:_model} for instance).
Non asymptotic bounds are obtained by \textcite{Maugis} even when the true density is not a Gaussian mixture.
All these works rely heavily on a \emph{bracketing} entropy analysis of
the models, that will also be central in our analysis.

When there is a covariate, the most classical extension of this model
is the Gaussian mixture regression, in which the means $\upsilon_k$
are now functions, is well studied as  described
in\citet{FiniteMixture}. Models in which the proportions vary have
been considered by \textcite{Antoniadis}.
Using idea of \textcite{Kolaczyk},  they have considered a model in which only proportion depend in a piecewise constant manner from the covariate. Their theoretical results are nevertheless obtained under
the strong assumption they exactly know the Gaussian components. This
assumption can be removed as shown by \citet{cohen12:_partit_based_condit_densit_estim}.
Models in which both mixture weights and means depend on the
covariate are considered by \textcite{Ge}, but in a logistic regression
mixture framework. They give conditions on the number of experts to obtain consistency of the posterior with logistic weights. Note that similar properties are studied by \textcite{Lee} for neural networks.

Although natural, Gaussian mixture regression with varying logistic weights seems to be mentioned first by
\textcite{HMEandEM}. They provide an algorithm similar to ours, based
on EM and IRLS, for hierarchical mixtures of experts but no
theoretical analysis.
 \textcite{Chamroukhi3} consider the case of piecewise polynomial
regression model with affine logistic weights. 
In our setting, this corresponds to a specific choice for $\Upsilon_K$
and $W_K$: a collection of piecewise polynomial and a set of affine
functions.
They use
a variation of the EM algorithm and a BIC criterion and provide
numerical experiments to support the efficiency of their scheme.
In this paper, we propose a slightly different  penalty choice and
prove non asymptotic bounds for the risk under very mild assumptions on $\Upsilon_K$
and $W_K$ that hold in their case.

\section{A model selection approach}

We will use a
model selection approach and define some conditional density models 
$S_m$ by specifying sets of Gaussian regression mixture conditional densities 
through
their number of classes $K$, a structure on the covariance matrices
$\Sigma_k$ and two function sets $\Upsilon_K$ and $W_K$ to which belong
respectively the $K$-tuple of means $(\upsilon_1,\ldots,\upsilon_K)$
and the $K$-tuple of logistic weights $(w_1,\ldots,w_K)$.
Typically those sets are compact subsets of polynomial of low
degree.
Within such a conditional density set $S_m$, we estimate $s$ by the maximizer $\widehat{s}_m$ of the likelihood
\begin{align*}
  \widehat{s}_{m} = \mathop{\text{argmax}}_{s_{K,\upsilon,\Sigma,w} \in S_m} \sum_{i=1}^{n} \ln s_{K,\upsilon,\Sigma,w}(Y_i|X_i),
\end{align*}
or more precisely, to avoid any existence issue, by any
$\eta$-minimizer of the -log-likelihood:
\begin{align*}
  \sum_{i=1}^{n} -\ln \widehat{s}_{m}(Y_i|X_i) \leq
  \min_{s_{K,\upsilon,\Sigma,w} \in S_m} \sum_{i=1}^{n} -\ln
  s_{K,\upsilon,\Sigma,w}(Y_i|X_i) + \eta.
\end{align*}
Assume now we have a collection $\{S_{m}\}_{m\in\mathcal{M}}$ of
models, for instance with different number of classes $K$ or different
maximum degree for the polynomials defining $\Upsilon_K$ and $W_K$, we
should choose the best model within this collection.
 Using only
the log-likelihood is not sufficient since this
favors models with large complexity. To balance this issue, we will
define a penalty $\pen(m)$ and select the model $\widehat{m}$ that
minimizes (or rather $\eta'$-almost minimizes) the sum of the
opposite of the log-likelihood and this penalty:
\begin{align*}
  \sum_{k=1}^{K} -\ln \widehat{s}_{\widehat{m}}(Y_i|X_i) 
+ \pen(\widehat{m}) 
\leq \min_{m\in\mathcal{M}} \sum_{k=1}^{K} -\ln \widehat{s}_{m}(Y_i|X_i) 
+ \pen(m) + \eta'.
\end{align*}

Our goal is now to define a penalty $\pen(m)$ which ensures that the
maximum likelihood estimate in the selected model performs almost as
well as the maximum likelihood estimate in the best model. More precisely, we
will prove that
\begin{align*}
\mathbb{E}\left[\JKL_{\rho}^{\otimes n}(s_0,\widehat{s}_{\widehat{m}})\right]
\leq C_1  \inf_{m \in \mathcal{M}} \left(\inf_{s_m \in S_m} \KL^{\otimes n}(s_0,s_m) + \frac{\pen(m)}{n}+ \frac{\eta + \eta'}{n} \right)
 +\frac{C_2}{n}
\end{align*}
where $\KL^{\otimes n}$ is a \emph{tensorized} Kullback-Leibler divergence,
$\JKL_{\rho}^{\otimes n}$ a lower bound of this divergence with a
$\pen(m)$ chosen of the same order as the variance of the
corresponding single model maximum likelihood estimate. 
In the next section, we specify all those divergences and explain the
general framework proposed by \textcite{ConditionalDensity} for conditional
density estimation. We will then explain how to use those results in
our specific setting. The last section is dedicated to some numerical
experiments conducted for sake of simplicity in the case where
$X\in[0,1]$ and $Y\in\mathbb{R}$.

\section{A general conditional density model selection theorem}

We summarize in this section the main result of \textcite{ConditionalDensity}
that will be our main tool to obtain the previous oracle
inequality. In this work, the estimator loss is measured with a
divergence $\JKL^{\otimes n}$ defined as a tensorized Kullback-Leibler divergence between the
true density and a convex combination of the true density and the
estimated one. Contrary to the true Kullback-Leibler divergence, to
which it is closely related, it is bounded. This boundedness turns out
to be crucial to control the loss of the penalized maximum likelihood
estimate under mild assumptions on the complexity of the model and
their collection.

Let $\KL$ be the classical Kullback-Leibler divergence, which measures a \emph{distance} between two density functions. Since we work
in a conditional density framework, we use a \emph{tensorized} version
of it. We define by 
$\KL^{\otimes n}$ the Kullback-Leibler tensorized divergence,
\begin{align*}
 \KL^{\otimes n}(s,t)=\mathbb{E} \left[ \frac{1}{n} \sum_{i=1}^n \KL(s(.|X_i),t(.|X_i)) \right]
\end{align*}
which appears naturally in this setting. Replacing $t$ by a convex
combination between $s$ and $t$ yields the so-called Jensen-Kullback-Leibler tensorized
divergence, denoted $\JKL_{\rho}^{\otimes n}$,
\begin{align*}
 \JKL_{\rho}^{\otimes n}(s,t)=\mathbb{E} \left[ \frac{1}{n}
   \sum_{i=1}^n \frac{1}{\rho} \KL(s(.|X_i),(1-\rho)s(.|X_i)+\rho
   t(.|X_i)) \right]
\end{align*} with $\rho \in ]0;1[$. This loss is always bounded by
$\frac{1}{\rho} \ln \frac{1}{1-\rho}$ but behaves as $\KL$ when $t$ is
close to $s$. Furthermore $\JKL_{\rho}^{\otimes n}(s,t)\leq
\KL_{\rho}^{\otimes n}(s,t)$. If we let $d^{2\otimes n}$ be the
tensorized extension of the squared Hellinger distance $d^2$, \textcite{ConditionalDensity} prove that there is a constant $C_{\rho}$ such that
$C_{\rho} d^{2\otimes n}(s,t)\leq \JKL_{\rho}^{\otimes n}(s,t)$.

To any model $S_{m}$, a set of conditional densities, we associate a
complexity defined in term of a
specific entropy, the bracketing entropy with respect to the root of $d^{2\otimes n}$. Recall that
a bracket $[t^-,t^+]$ is a pair of real functions such that $\forall
(x,y) \in \mathcal{X} \times \mathcal{Y}, t^-(x,y)\leq t^+(x,y)$ and a function $s$ is said to belong to the bracket $[t^-,t^+]$ if 
$\forall (x,y) \in \mathcal{X} \times \mathcal{Y}, t^-(x,y)\leq
s(x,y)\leq t^+(x,y)$. The bracketing entropy $H_{[],d}(\delta,S)$ of a
set $S$ is defined as the logarithm of the minimal number $N_{[],d}(\delta,S)$ of
brackets $[t^-,t^+]$ covering $S$, such that $d(t^-,t^+)\leq\delta$.
Our main assumption on models is an upper bound of a Dudley type
integral of these bracketing entropies:
\begin{description}
\item[Assumption (H)] For every model $S_m$ in the collection $\mathcal{S}$, there is a non-decreasing function $\phi_m$ such that $\delta \mapsto \frac{1}{\delta} \phi_m(\delta)$ is non-increasing on ]0,+$\infty$[ and for every $\sigma \in \mathbb{R}^+$,
\begin{align*}
\int_0^{\sigma} \sqrt{H_{[.],d^{\otimes n}} (\delta,S_m)} d\delta \leq \phi_m(\sigma).
\end{align*}
\end{description}

One need further to control the complexity of the collection as a
whole through a coding type (Kraft) assumption.
\begin{description}
\item[Assumption (K)] There is a family $(x_m)_{m \in \mathcal{M}}$ of non-negative numbers such that
\[
\sum_{m \in \mathcal{M}} e^{-x_m} \leq \Xi <+\infty.
\]
\end{description}
For technical reason, a separability assumption, always satisfied in
the setting of this paper, is also required.
\begin{description}
\item[Assumption (Sep)] For every model $S_m$ in the collection $\mathcal{S}$, there exists some countable subset $S'_m$ of $S_m$ and a set $\mathcal{Y'}_m$ with $\lambda(\mathcal{Y}\backslash \mathcal{Y}'_m)=0$ such that for every $t$ in $S_m$, it exists some sequence $(t_k)_{k\geq 1}$ of elements of $S'_m$ such that for every $x$ and every $y \in \mathcal{Y}'_m, \ln(t_k(y|x)) \xrightarrow[k \rightarrow + \infty]{} \ln(t(y|x))$.
\end{description}

The main result of \textcite{ConditionalDensity} is a condition on the penalty
$\pen(m)$ which ensures an oracle type inequality:
\begin{thm}
\label{thm:general}Assume we observe $(X_i,Y_i)$ with unknown conditional density $s_0$.
Let $\mathcal{S}=(S_m)_{m \in \mathcal{M}}$ an at most countable
conditional density model collection. 
Assume assumptions (H), (Sep) and (K) hold.
Let $\widehat{s}_m$ be a $\eta$ -log-likelihood minimizer in
  $S_m$
\[
\sum_{i=1}^n - \ln(\widehat{s}_m(Y_i|X_i) ) \leq
\inf_{s_m \in
  S_m} \left(  \sum_{i=1}^n - \ln(s_m(Y_i|X_i) ) \right)
+ \eta
\]

Then for any $\rho\in(0,1)$ and any $C_1>1$, there is a constant
$\kappa_0$ depending only on $\rho$ and
$C_1$ such that,
as soon as for every index $m \in \mathcal{M}$,
\[
pen(m) \geq \kappa (n\sigma^2_m+x_m)
\]
 with $\kappa >\kappa_0$
and $\sigma_m$ the unique root of
$\frac{1}{\sigma} \phi_m(\sigma)=\sqrt{n}\sigma$,
the penalized likelihood estimate $\widehat{s}_{\widehat{m}}$
with $\widehat{m}$ such that 
\[
\sum_{i=1}^n - \ln(\widehat{s}_{\widehat{m}}(Y_i|X_i) )+pen(\widehat{m}) \leq 
\inf_{m \in
  \mathcal{M}} \left(  \sum_{i=1}^n - \ln(\widehat{s}_m(Y_i|X_i) )+pen(m) \right)
+ \eta'
\]
satisfies
\begin{align*}
&\mathbb{E}\left[\JKL_{\rho}^{\otimes
    n}(s_0,\widehat{s}_{\widehat{m}})\right]\\
&\qquad\quad
\leq C_1  \inf_{m \in \mathcal{M}} \left(\inf_{s_m \in S_m}
  \KL_{\lambda}^{\otimes n}(s_0,s_m) + \frac{pen(m)}{n}\right) + C_1 \frac{\kappa_0\Xi+\eta + \eta'}{n}.
\end{align*}
\end{thm}
The name oracle type inequality means that the right-hand side is a
proxy for the estimation risk of the best model within the collection.
The term $\inf_{s_m \in S_m} \KL_{\lambda}^{\otimes n}(s_0,s_m)$ is a
typical bias term while $\frac{pen(m)}{n}$ plays the role of the
variance term. We have three sources of loss here: the constant
$C_1$ can not be taken equal to $1$, we use a different divergence on
the left and on the right and $\frac{pen(m)}{n}$ is not directly
related to the variance. The first issue is often considered as minor while
the second one turns out to be classical in density estimation
results. Whenever $pen(m)$ can be chosen approximately proportional to
the dimension $D_m$ of the model, which will be the case in our
setting,  $\frac{pen(m)}{n}$ is approximately proportional to
$D_m/n$, which is the asymptotic variance in the parametric case.
The right-hand side matches nevertheless the best known bound obtained
for a single model within such a general framework.

In the next section, we show how to apply this result in our Gaussian
mixture setting and prove that the penalty can be chosen roughly proportional
to the intrinsic dimension of the model, and thus of the order of the variance. 

\section{Spatial Gaussian regression mixture estimation theorem}

As explained in introduction, we are looking for conditional densities of type
\begin{align*}
s_{K,\upsilon,\Sigma,w}(y | x)= \sum_{k=1}^{K} \pi_{w,k}(x) \Phi_{\upsilon_k(x),\Sigma_k}(y),
\end{align*}
where $K \in \mathbb{N}^{*}$ is the number of mixture components,
$\Phi_{\upsilon,\Sigma}$ is the density of a Gaussian of mean $\upsilon$ and
covariance matrix $\Sigma$, 
 $\upsilon_{k}$ is a function specifying the
mean given $x$ of the $k$-th component while $\Sigma_k$ is its
covariance matrix and  the mixture weights $\pi_{w,k}$ are defined from a
collection of $K$ functions $w_1,\ldots,w_K$ by a logistic scheme:
\begin{align*}
  \pi_{w,k}(x) = \frac{e^{w_k(x)}}{\sum_{k'=1}^{K} e^{w_{k'}(x)}}.
\end{align*}
For sake of simplicity, we will assume that the covariate $X$ belongs
to an hypercube so that $\mathcal{X}=[0;1]^d$.

We will estimate those conditional densities by conditional densities
belonging to some model $S_m$
defined by
\begin{align*}
S_m =&\bigg\{(x,y)\mapsto\sum_{k=1}^{K}
  \pi_{w,k}(x)\Phi_{\upsilon_k(x),\Sigma_k}(y) \big|
(w_1,\ldots, w_K) \in W_K, (\upsilon_1,\ldots,\upsilon_K) \in \Upsilon_K, \\
 &(\Sigma_1,\ldots,\Sigma_K) \in V_K  \bigg \}
\end{align*}
where $W_K$ is a compact set of $K$-tuples of functions from
$\mathcal{X}$ to $\mathbb{R}$, $\Upsilon_K$ a compact set of $K$-tuples
of functions from $\mathcal{X}$ to $\mathbb{R}^p$ and $V_K$  a compact
set of $K$-tuples of covariance matrix of size $p\times p$.
Before describing more precisely those sets, we recall that $S_m$
will be taken in a model collection $\mathcal{S}=(S_m)_m$,
where $m$ specifies a choice for each of those parameters. The number
of components $K$ can be chosen arbitrarily in $\mathbb{N}^*$, but
will in practice and in our theoretical example be chosen smaller than an arbitrary $K_{\max}$, which may depend on the sample size $n$.
The sets $W_K$ and $\Upsilon_K$ will be typically chosen as a tensor product
of a same compact set of moderate dimension, for instance a set of
polynomial of degree smaller than respectively $\degW$ and $\degUpsilon$
whose coefficients are smaller in absolute values than respectively
$\maxaW$ and $\maxaUpsilon$. The structure of the set $V_K$ depends on the
\emph{noise} model chosen: we can assume, for instance, it is common to all
regressions, that they share a similar volume or diagonalization
matrix or they are all different. More precisely, we decompose any
covariance matrix $\Sigma$ into $LDAD'$, where $L=|\Sigma|^{1/p}$ is a
positive scalar corresponding to the volume, $D$ is the matrix of
eigenvectors of $\Sigma$ and $A$ the diagonal matrix of normalized
eigenvalues of $\Sigma$.
Let $L_-, L_+$ be positive values and $\lambda_-, \lambda_+$ real
values. We define the set $\mathcal{A}(\lambda_-,\lambda_+)$ of
diagonal matrices $A$ such that $|A|=1$ and $\forall  i\in
\{1,\ldots,p\}, \lambda_- \leq A_{i,i} \leq \lambda_+$.
A set $V_K$ is defined by
\begin{align*}
V_K=&\left\{(L_1D_1A_1D_1',\ldots,L_KD_KA_KD_K') |
\forall k, L_- \leq L_k \leq L_+,
D_k \in SO(p),\right.\\
& \left. A_k \in \mathcal{A}(\lambda_-,\lambda_+) \right\}
\end{align*}
Those sets $V_K$ correspond to the classical covariance matrix sets
described by \textcite{Celeux}.

We will bound the complexity term $n\sigma_m^2$ in term of the
\emph{dimension} of $S_m$: we prove that those two terms are roughly
proportional. The set $V_K$ is a parametric set and thus $\dim(V_K)$ is easily defined as the dimension of its parameter set.
Defining the dimension of $W_K$ and $\Upsilon_K$ is more interesting.
We rely on an entropy type definition of the dimension. For any
$K$-tuples of functions $(s_1,\ldots,s_K)$ and $(t_1,\ldots,t_K)$, we let
\begin{align*}
 d_{\|\sup\|_{\infty}}\left((s_1,\ldots,s_K),(t_1,\ldots,t_K)\right)
= \sup_{x\in\mathcal{X}} \sup_{1\leq k\leq K} |s_k(x)-t_k(x)|
\end{align*}
and define the dimension $\dim(F_K)$ of a set $F_K$ of such $K$-tuples as the
smallest $D$ such that there is a $C$ satisfying
\begin{align*}
  H_{d_{\|\sup\|_{\infty}}}(\sigma,F_K) \leq D \left( C + \ln
    \frac{1}{\sigma} \right).
\end{align*}

Using the following proposition of \textcite{ConditionalDensity}, we can easily verify that Assumption (H) is satisfied.

\begin{prop} 
\label{propsigma_m}
If for any $\delta \in [0;\sqrt{2}], H_{[.],d^{\otimes n}} (\delta,S_m) \leq D_m (C_m+\ln(\frac{1}{\delta}))$,
then the function $\phi_m(\sigma)=\sigma \sqrt{D_m} \left(\sqrt{C_m}+\sqrt{\pi}+\sqrt{\ln(\frac{1}{\sigma \wedge 1})}\right)$ satisfies assumption (H).
Furthermore, the unique root $\sigma_m$ of $\frac{1}{\sigma} \phi_m(\sigma)=\sqrt{n}\sigma$ satisfies 
\begin{align*}
n\sigma_m^2 \leq D_m \left( 2(\sqrt{C_m}+\sqrt{\pi})^2+\left(\ln \frac{n}{(\sqrt{C_m}+\sqrt{\pi})^2D_m}\right)_+ \right).
\end{align*}
\end{prop}

We show in Appendix that if
\begin{align*}
  H_{d_{\|\sup\|_{\infty}}}(\sigma,W_K) \leq \dim(W_K) \left( C_{W_K} + \ln
    \frac{1}{\sigma} \right)
\end{align*}
and
\begin{align*}
  H_{\max_k \sup_x \|\|_2}(\sigma,\Upsilon_K) \leq \dim(\Upsilon_K) \left( C_{\Upsilon_K} + \ln
    \frac{1}{\sigma} \right)
\end{align*}
then, if $n\geq1$, the complexity of the corresponding model $S_m$ satisfies
\begin{align*}
n\sigma_m^2 &\leq D_m \left( 2(\sqrt{C_m}+\sqrt{\pi})^2+\left(\ln \frac{n}{(\sqrt{C_m}+\sqrt{\pi})^2D_m}\right)_+ \right)\\
&\leq D_m \left( 2(\sqrt{C_m}+\sqrt{\pi})^2 +\ln(n) \right)\\
&\leq D_m (C'_m+\ln(n))   
\end{align*}
with $C'_m$ that depends only on the constants defining $V_K$ and the
constants $C_{W_K}$ and $C_{\Upsilon_K}$.
In order to obtain the same constant $C'_m$ for all models, we impose
that the dimension bound holds with the same constants for all models:
\begin{description}
\item[Assumption (DIM)] There exist two constants $C_W$ and
  $C_{\Upsilon}$ such that,
for every model $S_m$ in the collection
  $\mathcal{S}$, 
\begin{align*}
H_{\max_k \|\|_{\infty}}(\sigma,W_K) \leq \dim(W_K) \left( C_{W} + \ln
    \frac{1}{\sigma} \right).
\end{align*}
and
\begin{align*}
  H_{\max_k \sup_x \|\|_2}(\sigma,\Upsilon_K) \leq \dim(\Upsilon_K) \left( C_{\Upsilon} + \ln
    \frac{1}{\sigma} \right)
\end{align*}
\end{description}

We can now state our main result:
\begin{thm}
\label{theo:selectdim}
For any collection of  Gaussian regression
mixtures satisfying (K) and (DIM), there is a constant $C$ such that
for any $\rho\in(0,1)$ and any $C_1>1$, there is a constant
$\kappa_0$ depending only on $\rho$ and
$C_1$ such that,
as soon as for every index $m \in \mathcal{M}$,
$pen(m) = \kappa ((C+ \ln n) \dim(S_m) +x_m)$ with $\kappa >\kappa_0$,
the penalized likelihood estimate $\widehat{s}_{\widehat{m}}$
with $\widehat{m}$ such that 
\[
\sum_{i=1}^n - \ln(\widehat{s}_{\widehat{m}}(Y_i|X_i) )+pen(\widehat{m}) \leq 
\inf_{m \in
  \mathcal{M}} \left(  \sum_{i=1}^n - \ln(\widehat{s}_m(Y_i|X_i) )+pen(m) \right)
+ \eta'
\]
satisfies
\begin{align*}
&\mathbb{E}\left[\JKL_{\rho}^{\otimes n}(s_0,\widehat{s}_{\widehat{m}})\right]\\
&\qquad\quad\leq C_1  \inf_{m \in \mathcal{M}} \left(\inf_{s_m \in
    S_m} \KL_{\lambda}^{\otimes n}(s_0,s_m) + \frac{\pen(m)}{n} + \frac{\kappa_0\Xi+\eta + \eta'}{n} \right).
\end{align*}
\end{thm}
In the previous theorem, the assumption on $\pen(m)$ could be replaced
by the milder one
\[
\pen(m) \geq \kappa \left(
2D_m C^2+D_m\left(\ln \frac{n}{C^2D_m}\right)_+ +x_m\right).
\]
To minimize arbitrariness, 
$x_m$ should be chosen such that $\frac{2 \kappa x_m}{\pen(m)}$ is as
small as possible. Notice that the constant $C$ only depends on the
model collection parameters, for instance on the maximal number of
components $K_{\max}$. As often in model selection, the collection may
be chosen according to to the sample size $n$. If the constant $C'$
grows no faster than $\ln(n)$, the penalty shape can be kept intact
and a similar result holds uniformly in $n$ up to a slightly larger
$\kappa_0$. For instance, as $K_{\max}$ only appears in $C$ through a
logarithmic term, $K_{\max}$ may grow as a power of the sample size.

We postpone the proof of this theorem to the Appendix and focus on
Assumption (DIM). This assumption can often be verified when the functions sets $W_K$
and $\Upsilon_K$ are defined as images of a finite dimensional compact
subset of parameters when $X\in[0,1]^d$. For example, those sets can be
defined as linear combination of a finite set of bounded functions whose
coefficients belong to a compact set. We study here the case of linear
combination of the first elements of a polynomial basis but similar
results hold, up to some modification on the coefficient sets, for many other
choices (first elements of a Fourier, spline or wavelet basis,
elements of an arbitrary bounded dictionary...)

Let $\degW$ and $\degUpsilon$ be two integers and $\maxaW$ and $\maxaUpsilon$ some positive numbers. We define
\begin{align*}
W&=\left\{w :[0;1]^d\rightarrow \mathbb{R}| w(x)=\sum_{|r|=0}^{\degW} \alpha_r x^r \mbox{ and } \|\alpha\|_{\infty}\leq \maxaW\right\}\\
\Upsilon&=\left\lbrace
\upsilon:[0;1]^d \rightarrow \mathbb{R}^p \Big| \forall j \in \{1,\ldots,p\}, \forall x, \upsilon_j(x)=\sum_{|r|=0}^{\degUpsilon} \alpha_r^{(j)} x^r \mbox{ and } \|\alpha\|_\infty\leq \maxaUpsilon
\right\rbrace
\end{align*} 
Let $W_K = \{0\} \times W^{K-1}$ and $\Upsilon_K = \Upsilon^K$.

We prove in Appendix that
\begin{lem}
\label{entropieAppli}
$W_K$ and $\Upsilon_K$ satisfy assumption (DIM), with $
C_W=\ln
\left(\sqrt{2}+ \maxaW \binom{\degW+ d}{d} \right)$ and 
$C_\Upsilon=\ln \left(\sqrt{2}+\sqrt{p}\binom{\degUpsilon +d}{d}\maxaUpsilon\right)$, not depending on $K$.
\end{lem}
To apply Theorem~\ref{theo:selectdim}, it remains to describe a
collection $(S_m)$ and a suitable choice for $(x_m)$. Assume, for
instance, that the models in our collection are defined by an arbitrary
maximal number of components $K_{\max}$, a common free structure for the covariance
matrix $K$-tuple and a common maximal degree for the sets $W_K$ and
$\Upsilon_K$, then one can verify that $\dim(S_m)=(K-1+Kp)
\binom{\degW+d}{d}+Kp\frac{p+1}{2}$ and that the weight family $(x_m = K)$
satisfy Assumption (K) with $\Xi\leq 1/(e-1)$. Theorem~\ref{theo:selectdim} yields then an
oracle inequality with $\pen(m)=\kappa \left((C+ \ln(n)) \dim(S_m)+x_m
\right)$. Note that as $x_m \ll (C +\ln(n)) \dim(S_m)$, one can obtain a
similar oracle inequality with $\pen(m)=\kappa (C+ \ln(n)) \dim(S_m)$ for
a slightly larger $\kappa$. Finally, as explained in the proof,
choosing a covariance structure from the finite collection of
\textcite{Celeux} or choosing the maximal degree for the sets $W_K$ and
$\Upsilon_K$ among a finite family can be obtained with the same
penalty but with a larger constant $\Xi$ in Assumption (K).

\section{Numerical scheme and numerical experiment}

We illustrate our theoretical result in a setting similar to
the one considered by~\citet{Chamroukhi3}. We observe $n$ pairs
$(X_i,Y_i)$ with $X_i\in[0,1]$ and $Y_i\in\mathbb{R}$ and look for the
best estimate of the conditional density $s_0(y|x)$ that can be
written
\begin{align*}
s_{K,\upsilon,\Sigma,w}(y | x)= \sum_{k=1}^{K} \pi_{w,k}(x) \Phi_{\upsilon_k(x),\Sigma_k}(y),
\end{align*}
with $w \in W_K$ and $\upsilon \in \Upsilon_K$. We consider the simple
case where $W_K$ and $\Upsilon_K$ comprise linear functions. We do not impose any structure on the covariance matrices. 
 Our aim is to estimate the \emph{best} number of components
$K$, as well as the model parameters. As described with more details
later, we use an EM type algorithm to
estimate the model parameters for each $K$ and select one using the
penalized approach described previously.

\begin{figure}
\centering
\subfloat[2 000 data points of example P]{
  \centering
  \includegraphics[width=.49\textwidth]{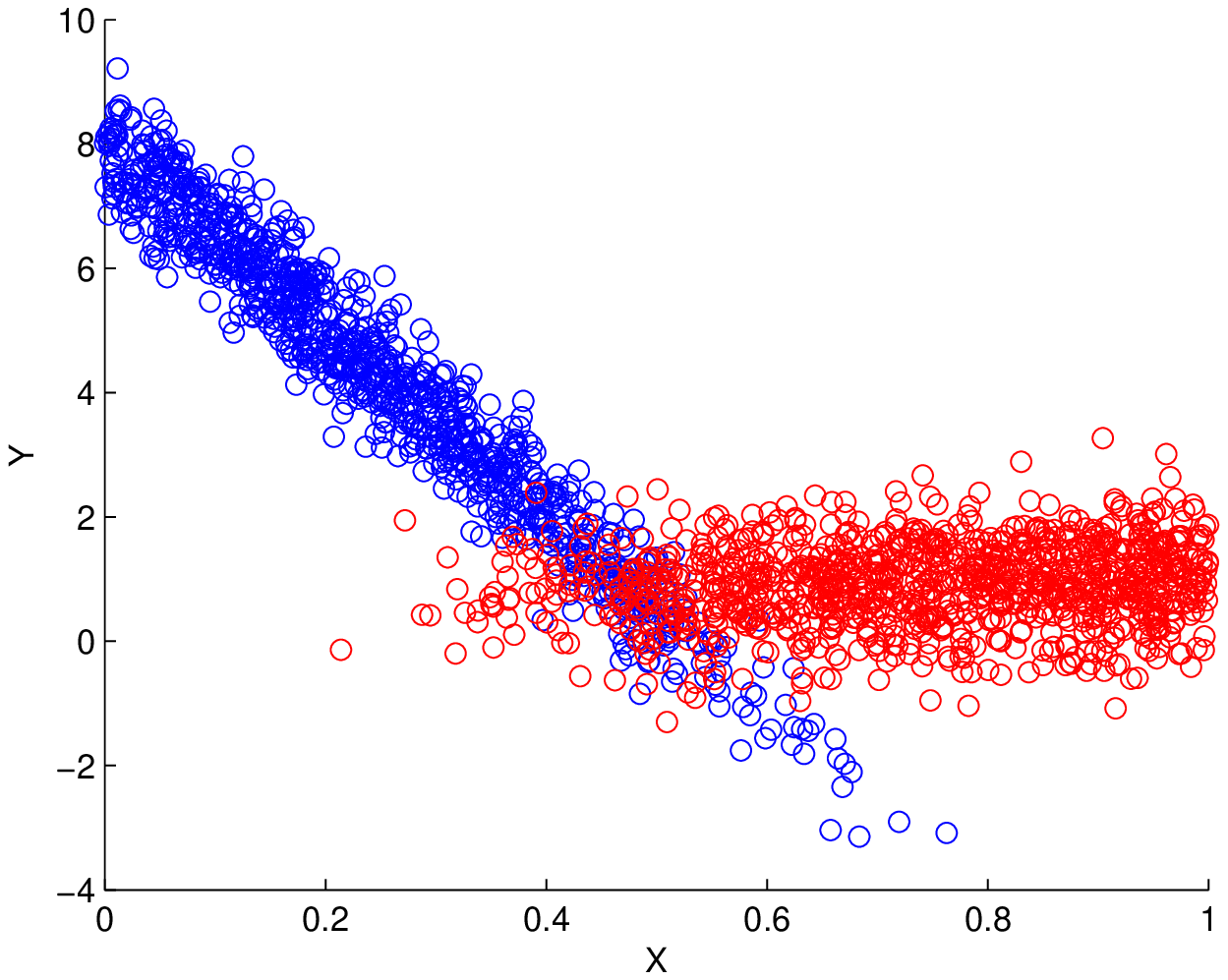}
}
\subfloat[2 000 data points of example NP]{
  \centering
  \includegraphics[width=.49\textwidth]{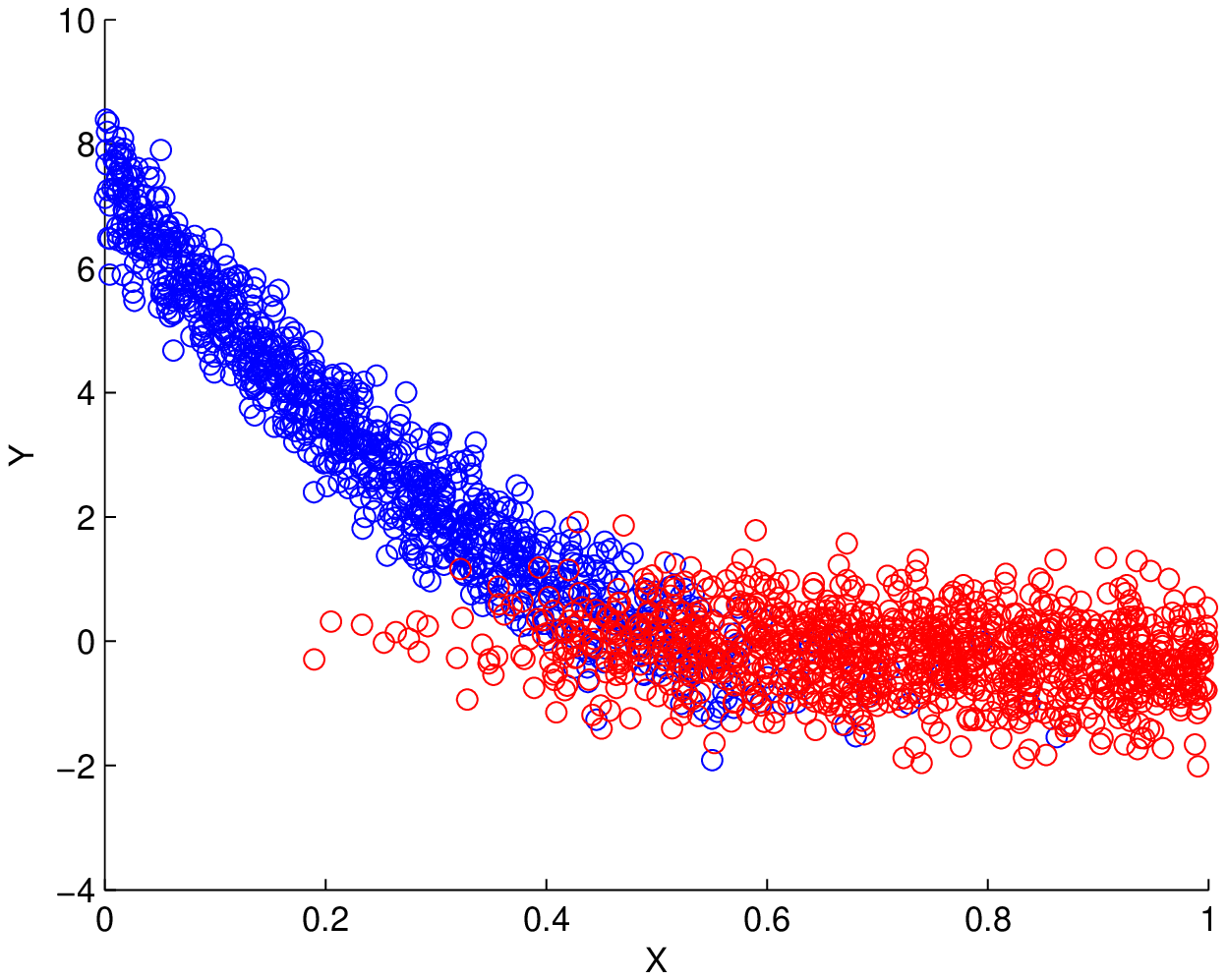}
}
\caption{Typical realizations}
\label{fig:donnees}
\end{figure}
In our numerical experiment, we consider two different examples: one
in which true conditional density belongs to one of our models, a
\emph{parametric} case, and one
in which this is not true, a \emph{non parametric} case. In the first situation, we expect to
perform almost as well as the maximum likelihood estimation in the
true model. In the second situation, we expect our algorithm to
automatically balance the model bias and its variance.
More precisely, we let
 \[s_0(y | x)= \frac{1}{1+\exp(15x-7)}
 \Phi_{-15x+8,0.3}(y)+\frac{\exp(15x-7)}{1+\exp(15x-7)}
 \Phi_{0.4x+0.6,0.4}(y)\]
in the first example, denoted example P,
and 
\[s_0(y | x)= \frac{1}{1+\exp(15x-7)}
\Phi_{15x^2-22x+7.4,0.3}(y)+\frac{\exp(15x-7)}{1+\exp(15x-7)}
\Phi_{-0.4x^2,0.4}(y)\]
in the second example, denoted example NP. 
For both experiments, we let $X$ be uniformly
distributed over $[0,1]$. Figure~\ref{fig:donnees} shows a typical realization
for both examples.

As often in model selection approach, 
the first step is to compute the maximum likelihood estimate for each
number of components $K$. To this purpose, we use a numerical scheme
based on the EM algorithm~\citep{AlgoEM} similar to the one used by
\citet{Chamroukhi3}. The only difference with a classical EM is in the
Maximization step since there is no closed formula for the weights
optimization. We use instead a Newton type algorithm. 
Note that we only perform a few  Newton steps
(5 at most) and ensures that the likelihood does not decrease.
 We have noticed that
there is no need to fully optimize at each step: we did not observe a
better convergence and the algorithmic cost is high. We denote from
now on this algorithm \emph{Newton-EM}.
Figure~\ref{fig:Lvit} illustrates the fast convergence of this
algorithm towards a local maximum of the likelihood.
\begin{figure}
\centering
\includegraphics[width=0.5\textwidth]{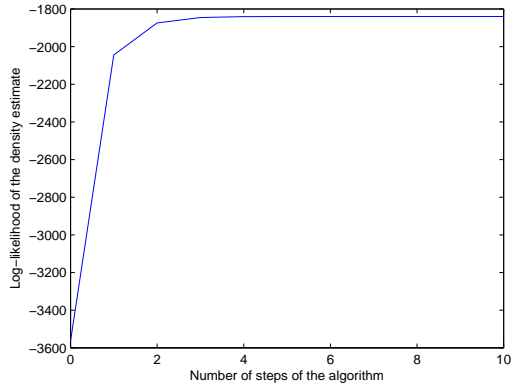}
\caption{Increase of the Log-likelihood of the estimated density 
at each step of our iterative \emph{Newton-EM} algorithm in the example
NP with 3 components and 2 000 data points.}
\label{fig:Lvit}
\end{figure}
Notice that the lower bound on the variance required in our theorem
appears to be necessary in practice. It avoids the spurious local
maximizer issue of EM algorithm, in which a class degenerates to a
minimal number of points allowing a perfect Gaussian regression fit.
We use a lower bound of $\frac{10}{n}$.
\citet{VarianceBound} provide a more precise data-driven bound:
$\frac{\min_{1\leq i<j \leq n}(Y_i-Y_j)^2}{2 \chi^2_{n-2K+1}((1-\alpha)^{1/K})}$, with $\chi^2_{n-2K+1}$ the chi-squared quantile function, which is of the same order as $\frac{1}{n}$ in our case.
In practice, the constant $10$ gave good results.

An even more important issue with EM algorithms is initialization, since the local minimizer obtained depends heavily on it. We observe that,
while the weights $w$ do not require a special care and can be
 simply initialized uniformly equal to
$0$, the means require much more attention in order to obtain a good
minimizer. We propose an initialization strategy which can be seen as an
extension of a Quick-EM scheme with random initialization.

We draw randomly $K$ lines, each defined as the line going through two
points $(X_i,Y_i)$ drawn at random among the observations. We perform
then a K-means clustering using the distance along the $Y$ axis. Our
\emph{Newton-EM} algorithm is initialized by the regression
parameters as well as the empirical variance on each of the $K$
clusters. We perform then $3$ steps of our minimization algorithm and
keep among $50$ trials the one with the largest likelihood. This
winner is used as the initialization of a final \emph{Newton-EM}
algorithm using 10 steps.

We consider two other strategies: a \emph{naive} one in which the
initial lines chosen at random and a common variance are used directly to initialize the
\emph{Newton-EM} algorithm and a \emph{clever} one in which
observations are 
first normalized in order to have a similar variance along both the
$X$ and the $Y$ axis, a K-means on both $X$ and $Y$ with $5$ times the
number of components is then performed and the initial lines are
drawn among the regression lines of the resulting cluster
comprising more than $2$ points.

The complexity of those procedures differs and as stressed by
\citet{Celeux} the fairest comparison is to perform them for the same
amount of time (5 seconds, 30 seconds, 1 minute...) and compare the
obtained likelihoods. The difference between the 3 strategies is not
dramatic: they yield very similar likelihoods. We nevertheless
observe that the \emph{naive} strategy has an important dispersion and
fails sometime to give a satisfactory answer. Comparison between  the
\emph{clever} strategy and the regular one is more complex since the
difference is much smaller. Following \citet{Celeux}, we have chosen
the regular one which corresponds to more random initializations and
thus may explores more local maxima.

Once the parameters' estimates have been computed for each $K$, we select the model that
minimizes
\begin{align*}
  \sum_{i=1}^n - \ln(\widehat{s}_m(Y_i|X_i) )+\pen(m)
\end{align*}
with $\pen(m)=\kappa \dim(S_m)$. Note that our theorem ensures that
there exists a $\kappa$ large enough for which the estimate has good
properties, but does not give an explicit value for $\kappa$. In
practice, $\kappa$ has to be chosen. The two most classical choices
are $\kappa=1$ and $\kappa=\frac{\ln n}{2}$ which correspond to the AIC and
BIC approach, motivated by asymptotic arguments.
 We have used here the slope heuristic proposed
by Birg\'{e} and Massart and described for instance
in~\citet{Slope}. It consists in representing the dimension of the
selected model according to $\kappa$ (fig~\ref{fig:Pente}), and
finding $\hat{\kappa}$ such that if $\kappa <\hat{\kappa} $, the
dimension of the selected model is large, and reasonable otherwise. 
The slope heuristic prescribes then the use of $\kappa=2\hat{\kappa}$.
In both examples, we have noticed that the sample's size had no significant influence on the choice of $\kappa$, and that very often 1 was in the range of possible values indicated by the slope heuristic.
 According to this observation, we have chosen in both examples $\kappa=1$.
\begin{figure}
\centering
\subfloat[Example P with 2 000 points]{
  \centering
  \includegraphics[width=.49\textwidth]{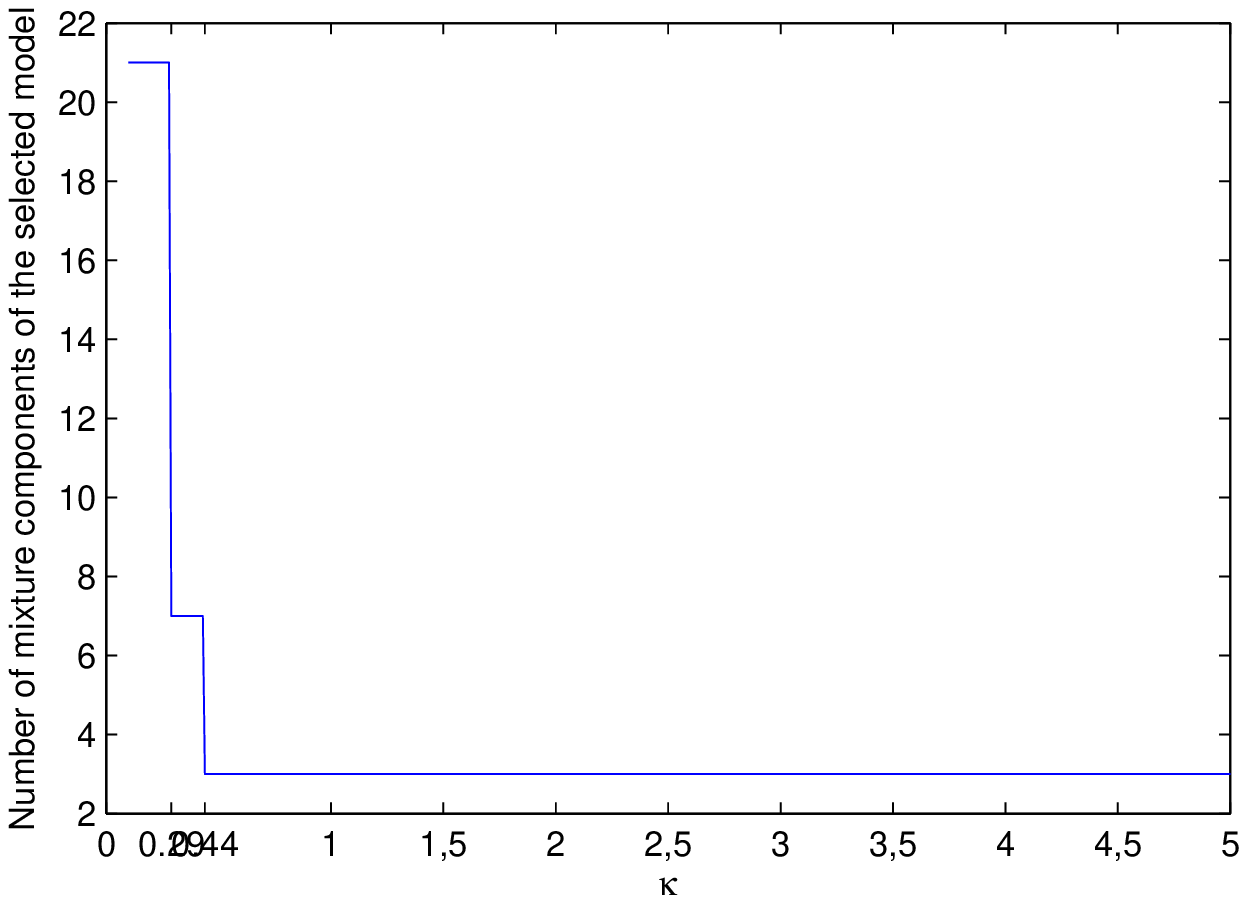}
  \label{fig:Pente1}
}
\subfloat[Example NP with 2 000 points]{
  \centering
  \includegraphics[width=.49\textwidth]{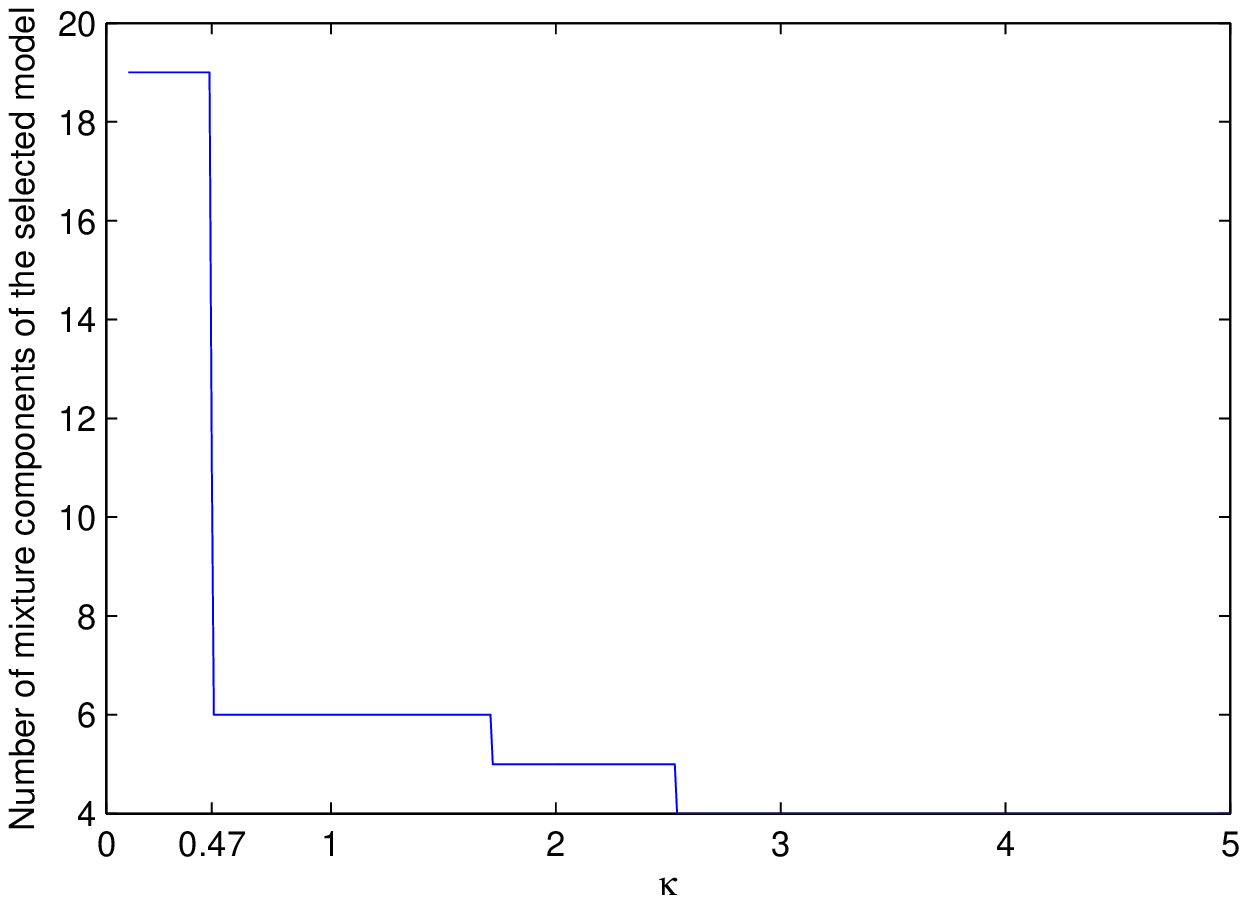}
  \label{fig:Pente2}
}
\caption{Slope heuristic: plot of the selected model dimension with
  respect to the penalty coefficient $\kappa$. In both examples,
  $\widehat{\kappa}$ is of order $1/2$.}
\label{fig:Pente}
\end{figure}

We measure performances in term of tensorized Kullback-Leibler distance. 
Since there is no known formula for tensorized Kullback-Leibler
distance in the case of Gaussian mixtures, and since we know the true
density, we evaluate the distance using Monte Carlo method. The
variability of this randomized evaluation has been verified to be negligible in practice.

For several numbers of mixture components and for the selected K, we
draw in figure~\ref{fig:Oracle} the box plots and the mean of
tensorized Kullback-Leibler distance over $55$ trials. The first
observation is that the mean of tensorized Kullback-Leibler distance
between 
the penalized estimator $\hat{s}_{\hat{K}}$ and $s_0$ is smaller than the mean
of tensorized Kullback-Leibler distance between $\hat{s}_{K}$ ans
$s_0$ over $K\in\{1,\ldots,20\}$. This is in line with the oracle type
inequality of Theorem~\ref{theo:selectdim}. Our numerical results
hint that our theoretical analysis may be pessimistic. A close
inspection show that the bias-variance trade-off differs between the
two examples. Indeed, since in the first one the true density belongs
to the model, the best choice is $K=2$ even for small $n$. As shown on
the histogram of Figure~\ref{fig:histo}, this is almost always the
model chosen by our algorithm. Observe also that the mean of
Kullback-Leibler distance seems to behave like  $\frac{\dim(S_m)}{2n}$
(shown by a dotted line). This is indeed the expected behavior when
the true model belongs to a nested collection and corresponds to the
classical AIC heuristic. In the second example, the true model does
not belong to the collection. The best choice for $K$ should thus
balance a model approximation error and a variance one. We observe in
Figure~\ref{fig:histo} such a behavior: the larger $n$ the more
complex the model and thus $K$. Note that the slope of the mean error
seems also to grow like $\frac{\dim(S_m)}{2n}$ even though there is no
theoretical guarantee of such a behavior.
 
\begin{figure}
\centering
\subfloat[Example P with 2 000 data points]{
\centering
\includegraphics[width=.49\textwidth]{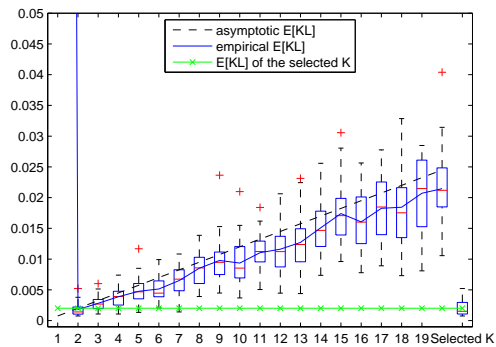}
}
\subfloat[Example P with 10 000 data points]{
\centering
\includegraphics[width=.49\textwidth]{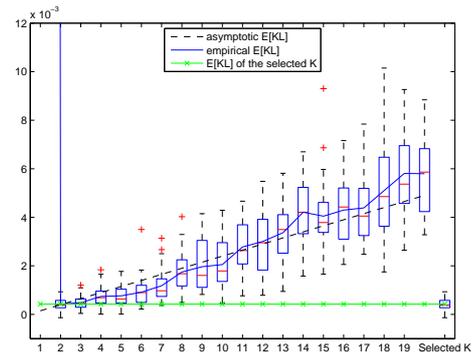}
}\\
\subfloat[Example NP with 2 000 data points]{
\centering
\includegraphics[width=.49\textwidth]{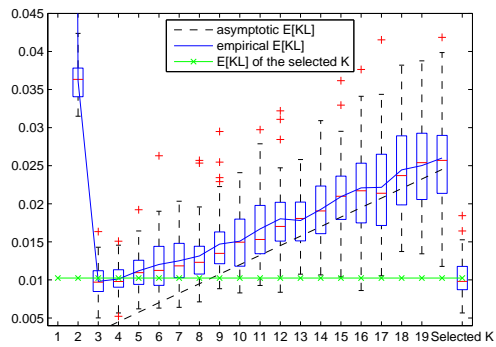}
}
\subfloat[Example NP with 10 000 data points]{
\centering
\includegraphics[width=.49\textwidth]{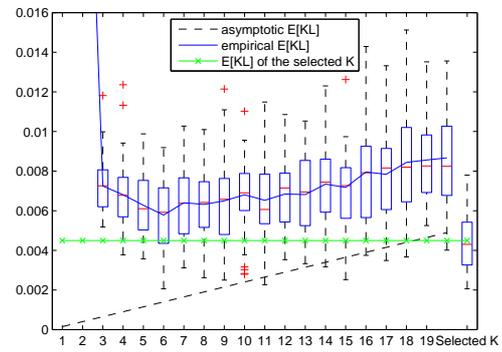}
 }
\caption{Box-plot of the Kullback-Leibler distance according to the number
  of mixture components. On each graph, the right-most box-plot shows
  this Kullback-Leibler distance for the penalized estimator $\hat{s}_{\widehat{K}}$}
\label{fig:Oracle}
\end{figure}

\begin{figure}
\centering
\subfloat[Example P with 2 000 data points]{
 \centering
\includegraphics[width=.49\textwidth]{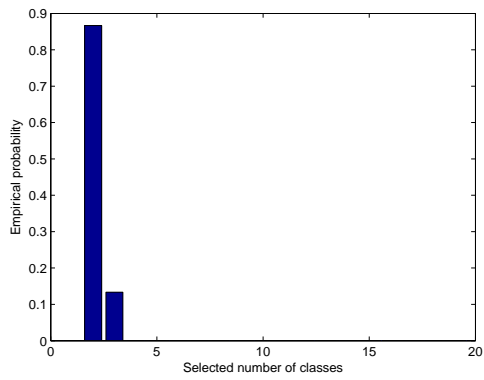}
}
\subfloat[Example P with 10 000 data points]{
\centering
\includegraphics[width=.49\textwidth]{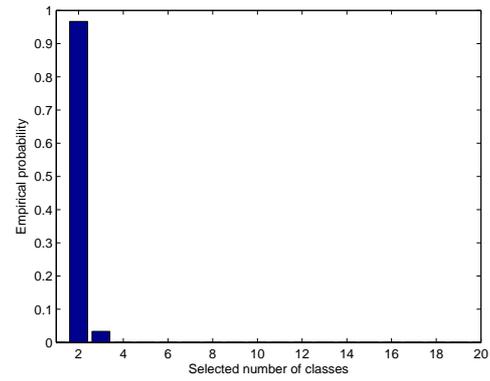}
}
\\
\subfloat[Example NP with 2 000 data points]{
\centering
\includegraphics[width=.49\textwidth]{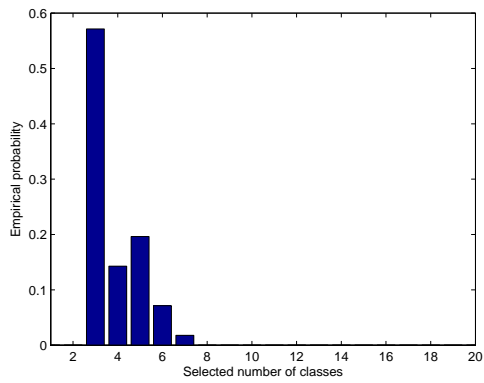}
}
\subfloat[Example NP with 10 000 data points]{

\centering
\includegraphics[width=.49\textwidth]{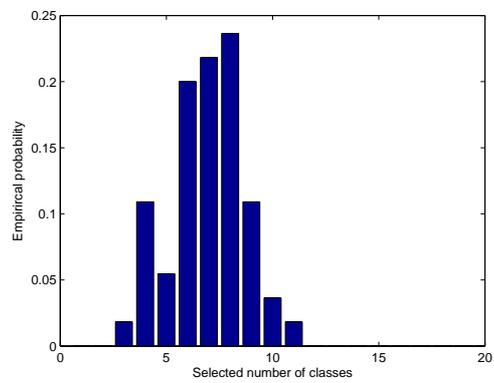}
}
\caption{Histograms of the selected K}
\label{fig:histo}
\end{figure}

Figure~\ref{fig:KLn} shows the error decay when the sample size $n$
grows. As expected in the parametric case, example P, we observe the decay in
$t/n$ predicted in the theory, with $t$ some constant. The rate in the second case appears to
be slower. Indeed, as the true conditional density does not belong to
any model, the selected models are more and more complex when $n$
grows which slows the error decay. In our theoretical analysis, this
can already be seen in the decay of the \emph{variance} term of the
oracle inequality. Indeed, if we let $m_0(n)$ be the optimal oracle
model, the one minimizing the right-hand side of the oracle
inequality, the variance term is of order $\frac{D_{m_0(n)}}{n}$ which
is larger than $\frac{1}{n}$ as soon as $D_{m_0(n)} \to +\infty$. It
is well known that the decay depends on the regularity of the
true conditional density. Providing a minimax analysis of the proposed
estimator, as have done~\citet{maugis12:_adapt_gauss}, would be
interesting but is beyond the scope of this paper.

\begin{figure}
\centering
\subfloat[Example P. The slope of the free regression line is
    $\simeq -1,3$]{
  \centering
  \includegraphics[width=.49\textwidth]
  {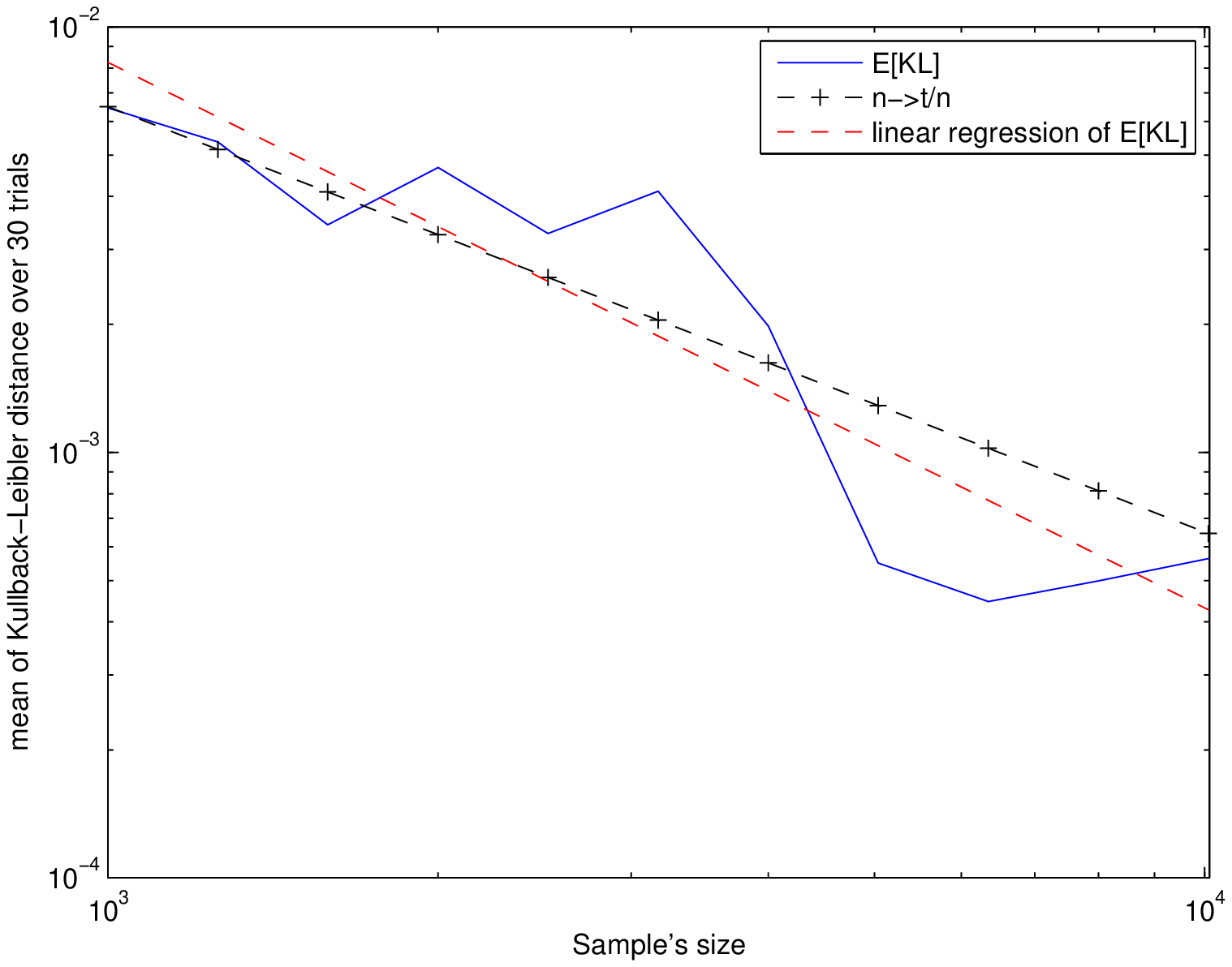}
  \label{fig:KLn1}
}
\subfloat[Example NP. The slope of the regression line is $\simeq -0,6$.]{
  \includegraphics[width=.49\textwidth]
{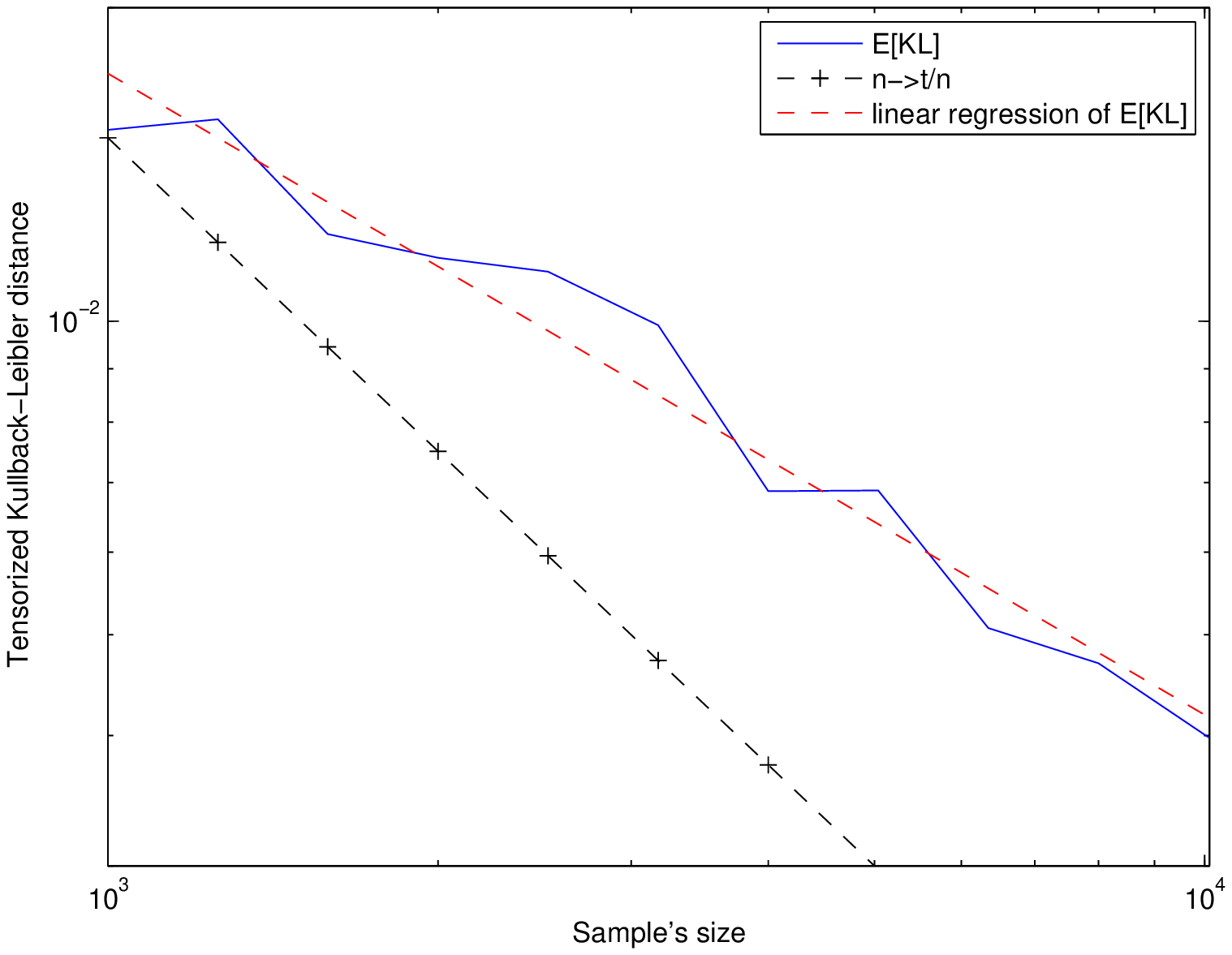}  
  \label{fig:KLn2}
}
\caption{Kullback-Leibler distance between the true density and the
  computed density using $(X_i,Y_i)_{i \leq N}$ with respect to the
  sample size, represented in a log-log scale. For each graph, we
  added a free linear least-square regression and one with slope $-1$
  to stress the two different behavior.}
\label{fig:KLn}
\end{figure}

\appendix

\section{Proof of Theorem~\ref{theo:selectdim}}
In this section, an overview of the proof of the model selection theorem, applied to our Gaussian regression mixture, is given. 
\ref{appendiceB} is dedicated to the example with polynomial means and weights. The constants in the Assumption (DIM) and the theorem are specified. 
Then, in \ref{appendiceC}, we provide more details on the proofs and lemmas used in the first section.

We will show that Assumption (DIM) ensures that
for all $\delta \in [0;\sqrt{2}],$ $H_{[.],d^{\otimes n}} (\delta,S_m) \leq D_m
(C_m+\ln(\frac{1}{\delta}))$ with a common $C_m$. If this happens,
Proposition~\ref{propsigma_m}
yields the results.
In other words, if we can control models' bracketing entropy with a
uniform constant $\mathfrak{C}$, we get a suitable bound on the complexity.
This result will be obtain by first decomposing the entropy term
between the weights and the Gaussian mixtures. Therefore we use the following distance over conditional densities:
\begin{align*}
\underset{x} \sup  \,  d_y(s,t)=\underset{x \in \mathcal{X}} \sup  \,  \left( \int_y \left(\sqrt{s(y|x)}-\sqrt{t(y|x)}\right)^2 dy \right)^{\frac{1}{2}}.
\end{align*}
Notice that $d^{2\otimes n}(s,t) \leq \sup_x d_y^2(s,t)$.

For all weights $\pi$ and $\pi'$, we define
\begin{align*}
\underset{x} \sup  \,  d_k(\pi,\pi')=\underset{x \in \mathcal{X}} \sup  \,  \left(\sum_{k=1}^K \left(\sqrt{\pi_k(x)}-\sqrt{\pi'_k(x)}\right)^2\right)^{\frac{1}{2}}.
\end{align*}
Finally, for all densities $s$ and $t$ over $\mathcal{Y}$, depending on $x$, we set
\begin{align*}
\underset{x} \sup  \, \underset{k} \max \,d_y(s,t)&=\underset{x \in \mathcal{X}} \sup  \, \underset{1\leq k \leq K} \max \,  d_y(s_k(x,.),t_k(x,.))\\
&=\underset{x \in \mathcal{X}} \sup  \, \underset{1\leq k \leq K} \max \, \left( \int_y \left(\sqrt{s_k(x,y)}-\sqrt{t_k(x,y)}\right)^2 dy \right)^{\frac{1}{2}}.
\end{align*}

\begin{lem}
Let $\mathcal{P}=\left \{(\pi_{w,k})_{1 \leq k \leq K} /w \in W_K,\mbox{ and }\forall (k,x), \pi_{w,k}(x)=\frac{e^{w_k(x)}}{\sum_{l=1}^K e^{w_l(x)}} \right \}$ and\\
$\mathcal{G}=\left \{\left(\Phi_{\upsilon_k,\Sigma_k} \right)_{1 \leq k \leq K}/\upsilon \in \Upsilon_K, \Sigma \in V_K \right \}$.
Then for all $\delta$ in [0;$\sqrt{2}$], for all $m$ in $\mathcal{M}$,
\[
H_{[.],\underset{x} \sup \, d_{y}}\left(\delta,S_{m}\right) \leq H_{[.],\underset{x} \sup  \,  d_k}\left(\frac{\delta}{5},\mathcal{P}\right) + H_{[.],\underset{x} \sup  \, \underset{k} \max \,  d_y}\left(\frac{\delta}{5},\mathcal{G}\right).
\]
\end{lem}

One can then relate the bracketing entropy of $\mathcal{P}$ to the
entropy of $W_K$
\begin{lem}
\label{controlePoids}
For all $\delta \in [0;\sqrt{2}]$,
\begin{align*}
H_{[.],\underset{x}\sup \, d_k}\left(\frac{\delta}{5},\mathcal{P}\right)&\leq
H_{\underset{k}\max \, \|\|_{\infty}}\left(\frac{3\sqrt{3}\delta}{20\sqrt{K}},W_K \right)
\end{align*}
\end{lem} 
Since $\mathcal{P}$ is a set of weights, $\frac{3\sqrt{3}\delta}{20\sqrt{K}}$ could be replaced by $\frac{3\sqrt{3}\delta}{20\sqrt{K-1}}$ with an identifiability condition. For example, $W'_K=\left\{(0,w_2-w_1,\ldots, w_K-w_1) |w \in W_K \right\}$ can be covered using brackets of null size on the first coordinate, lowering squared Hellinger distance between the brackets' bounds to a sum of $K-1$ terms. Therefore, $H_{[.],\underset{x}\sup \, d_k}\left(\frac{\delta}{5},\mathcal{P}\right)\leq
H_{\underset{k}\max \, \|\|_{\infty}}\left(\frac{3\sqrt{3}\delta}{20\sqrt{K-1}},W'_K \right)$.

 Since we have assumed that $\exists D_{W_K},C_W$ s.t $\forall \delta \in [0;\sqrt{2}]$,
\begin{align*}
H_{\underset{k}\max \, \|\|_{\infty}}\left(\delta,W_K \right)
&\leq D_{W_K} \left(C_W+\ln\left(\frac{1}{\delta}\right) \right)
\end{align*}
Then
\begin{align*}
H_{[.],\underset{x}\sup \, d_k}\left(\frac{\delta}{5},\mathcal{P}\right)&\leq
D_{W_K} \left(C_W+\ln\left(\frac{20\sqrt{K}}{3\sqrt{3}\delta}\right) \right)
\end{align*}

To tackle the Gaussian regression part, we rely heavily on 
the following proposition,
\begin{prop}
Let $\kappa\geq \frac{17}{29}$,
\( \displaystyle
  \gamma_\kappa  =
  \frac{25(\kappa-\frac{1}{2})}{49(1+\frac{2\kappa}{5})}
\). For any $0<\delta\leq\sqrt{2}$
 and any $\delta_\Sigma \leq
 \frac{1}{5\sqrt{\kappa^2 \cosh(\frac{2\kappa}{5}) +
    \frac{1}{2}}} \frac{\delta}{p}$, 
$(\upsilon,L,A,D) \in \Upsilon \times [L_-,L_+] \times \mathcal{A}(\lambda_-,\lambda_+) \times SO(p)$ and
$(\tilde{\upsilon},\tilde{L},\tilde{A},\tilde{D}) \in \Upsilon \times [L_-,L_+] \times \mathcal{A}(\lambda_-,+\infty) \times SO(p), \Sigma=LDAD'$ and  $\tilde{\Sigma}=\tilde{L}\tilde{D}\tilde{A}\tilde{D}'$, assume that
$t^-(x,y)=(1+\kappa \delta_\Sigma)^{-p} \Phi_{\tilde{\upsilon}(x),(1+\delta_\Sigma)^{-1}\tilde{\Sigma}}(y)$ and 
$t^+(x,y)=(1+\kappa \delta_\Sigma)^p
\Phi_{\tilde{\upsilon}(x),(1+\delta_\Sigma)\tilde{\Sigma}}(y)$.
   
If
\begin{align*}
\begin{cases}
\forall x \in \mathbb{R}^d, \|\upsilon(x)-\tilde{\upsilon}(x)\|^2 \leq p \gamma_{\kappa} L_- \lambda_-
  \frac{\lambda_-}{\lambda_+} {\delta_\Sigma}^2\\
(1+\frac{2}{25}\delta_\Sigma)^{-1} \tilde{L} \leq L \leq \tilde{L}\\
\forall 1\leq i \leq p, |A_{i,i}^{-1}-\tilde{A}_{i,i}^{-1}| \leq \frac{1}{10}
\frac{\delta_\Sigma}{\lambda_+} \\
 \forall y \in \mathbb{R}^p, \|Dy-\tilde{D}y\| \leq
\frac{1}{10} \frac{\lambda_-}{\lambda_+}
\delta_\Sigma \|y\|
\end{cases}
\end{align*}
then $[t^-,t^+]$ is a $\frac{\delta}{5}$ Hellinger bracket such that $t^-(x,y) \leq \Phi_{\upsilon(x),\Sigma}(y) \leq t^+(x,y)$.
\end{prop}

We consider three cases:
the parameter (mean, volume, matrix) is known ($\star=0$), unknown but common to all classes ($\star=c$), unknown and possibly different for every class ($\star=K$). For example, $[\nu_K, L_0, D_c, A_0]$ denotes a model in which only means are free and eigenvector matrices are assumed to be equal and unknown.
Under our assumption that $D_{\Upsilon_K}, C_{\Upsilon}$ s.t $\forall \delta \in [0;\sqrt{2}]$,
\begin{align*}
H_{\max_k \sup_x \|.\|_2}(\delta,\Upsilon_K)&\leq D_{\Upsilon_K} \left(C_{\Upsilon}+\ln\left(\frac{1}{\delta}\right) \right)
\end{align*}
we deduce:
\begin{align}
\label{controleGaussienne}
H_{[.],\max_k \sup_x d_y}\left(\frac{\delta}{5},\mathcal{G}\right)\leq 
\mathcal{D} \left(\mathcal{C}+\ln\left(\frac{1}{\delta} \right) \right)
\end{align}
where $\displaystyle\mathcal{D}=Z_{\upsilon,\star}+Z_{L,\star}+\frac{p(p-1)}{2}Z_{D,\star}+(p-1)Z_{A,\star}$ and 
\begin{align*}
\mathcal{C}&=\ln\left(5p\sqrt{\kappa^2 \cosh \left(\frac{2\kappa}{5}\right)+\frac{1}{2}}\right)+\frac{Z_{\upsilon,\star}C_{\Upsilon}}{\mathcal{D}} 
+\frac{Z_{\upsilon,\star}}{2\mathcal{D}} \ln\left(\frac{\lambda_+}{p\gamma_\kappa L_- \lambda_-^2}\right)\\
&
+\frac{Z_{L,\star}}{\mathcal{D}} \ln\left(\frac{4+129\ln\left(\frac{L_+}{L_-}\right)}{10}\right)
+\frac{Z_{D,\star}}{\mathcal{D}} \left( \ln(c_U)+\frac{p(p-1)}{2} \ln\left(\frac{10\lambda_+}{\lambda_-} \right)\right)\\
&+\frac{Z_{A,\star}(p-1)}{\mathcal{D}}
\ln\left(\frac{4}{5}+\frac{52\lambda_+}{5\lambda_-}\ln\left(\frac{\lambda_+}{\lambda_-}\right)\right).\\
Z_{\upsilon,K}&=D_{\Upsilon_K}, Z_{\upsilon,c}=D_{\Upsilon_1}, Z_{\upsilon,0}=0\\
Z_{L,0}&=Z_{D,0}=Z_{A,0}=0,\\
Z_{L,c}&=Z_{D,c}=Z_{A,c}=1,\\
Z_{L,K}&=Z_{D,K}=Z_{A,K}=K
\end{align*}

We notice that the following upper-bound of $\mathcal{C}$ is independent from the model of the collection, because we have made this hypothesis on $C_\Upsilon$.
\begin{align*}
\mathcal{C}&\leq\ln\left(5p\sqrt{\kappa^2 \cosh \left(\frac{2\kappa}{5}\right)+\frac{1}{2}}\right)+C_{\Upsilon} 
+\frac{1}{2} \ln\left(\frac{\lambda_+}{p\gamma_\kappa L_- \lambda_-^2}\right)\\
&+\ln\left(\frac{4+129\ln\left(\frac{L_+}{L_-}\right)}{10}\right)
+\frac{2}{p(p-1)}\ln(c_U)+\ln\left(\frac{10\lambda_+}{\lambda_-} \right)\\
&+\ln\left(\frac{4}{5}+\frac{52\lambda_+}{5\lambda_-}\ln\left(\frac{\lambda_+}{\lambda_-}\right)\right):=\mathcal{C}_1.
\end{align*}

We conclude that $H_{[.],\sup_x d_y}\left(\delta,S_m\right)\leq
D_m \left(C_m+\ln\left(\frac{1}{\delta}\right) \right)$, with
\begin{align*}
D_m&=D_{W_K}+\mathcal{D}\\
C_m&=\frac{D_{W_K}}{D_m}\left(C_W+\ln\left(\frac{20\sqrt{K}}{3\sqrt{3}}\right)\right)
+\frac{\mathcal{D}\mathcal{C}_1}{D_m}\\
&\leq C_W+\ln\left(\frac{20\sqrt{K_{\max}}}{3\sqrt{3}}\right)+\mathcal{C}_1:=\mathfrak{C}
\end{align*}
Note that the constant $\mathfrak{C}$ does not depend on the dimension $D_m$ of the model, thanks to the hypothesis that $C_W$ is common for every model $S_m$ in the collection.
Using Proposition \ref{propsigma_m}, we deduce thus that
\begin{align*}
  n \sigma_m^2 \leq D_m \left(2 \left(\sqrt{\mathfrak{C}}+\sqrt{\pi}\right)^2+ \left(\ln\frac{n}{\left(\sqrt{\mathfrak{C}}+\sqrt{\pi}\right)^2 D_m}\right)_+ \right).
\end{align*}
Theorem~\ref{thm:general} yields then, for a collection
$\mathcal{S}=\left(S_m\right)_{m\in \mathcal{M}}$, with
$\mathcal{M}=\lbrace (K,W_K,\Upsilon_K,V_K)| K\in \mathbb{N}^*,
W_K,\Upsilon_K,V_K \mbox{ as previously defined } \rbrace$ for which
Assumption (K) holds, the oracle inequality of
Theorem~\ref{theo:selectdim} as soon as
\begin{align*}
\pen(m)\geq \kappa \left(D_m \left(2
    \left(\sqrt{\mathfrak{C}}+\sqrt{\pi}\right)^2+
    \left(\ln\frac{n}{\left(\sqrt{\mathfrak{C}}+\sqrt{\pi}\right)^2
        D_m}\right)_+ \right) +x_m\right).
\end{align*}

\section{Proof of Theorem for polynomial}
\label{appendiceB}
We focus here on the example in which $W_K$ and $\Upsilon_K$ are
polynomials of degree respectively at most $\degW$ and $\degUpsilon$.

By applying lemmas~\ref{entropieAppli}, \ref{controlePoids} and \ref{controleGaussienne}, we get:
\begin{coro}
\begin{align*}
H_{[.],\underset{x}\sup \, d_k}\left(\frac{\delta}{5},\mathcal{P}\right)
&
\leq (K-1)\binom{\degW+d}{d}
\\&\qquad\quad\times
\left[\ln\left(\sqrt{2}+\frac{20}{3\sqrt{3}}\maxaW\sqrt{K-1}\binom{\degW+d}{d}\right)
+\ln\left(\frac{1}{\delta}\right) \right].\\
&\leq (K-1)\binom{\degW+d}{d}
\\&\qquad\quad\times
\left[C_W+\ln\left(\frac{20}{3\sqrt{3}}\sqrt{K-1}\right)
+\ln\left(\frac{1}{\delta}\right) \right].\\
H_{[.],\underset{x} \sup  \, \underset{k} \max \,  d_y}\left(\frac{\delta}{5},\mathcal{G}\right)&\leq 
 \mathcal{D} \left(\mathcal{C}+\ln \left(\frac{1}{\delta}\right) \right)
\end{align*}
with
\begin{align*}
\mathcal{D}&=D_{\Upsilon^K}+K \frac{p(p+1)}{2}, D_{\Upsilon^K}=pK \binom{\degUpsilon+d}{d}\\
\mathcal{C}&=\frac{2}{2D_{\Upsilon^K}+Kp(p+1)} \left(D_{\Upsilon^K}C_{\Upsilon}
+\frac{D_{\Upsilon^K}}{2} \ln\left(\frac{25 p \lambda_+ \left(\kappa^2 \cosh \left(\frac{2\kappa}{5}\right)+\frac{1}{2}\right)}{\gamma_\kappa L_- \lambda_-^2}\right)\right.\\
&\left. +K
\left[\ln(c_U)
+\ln\left(\frac{4+129\ln\left(\frac{L_+}{L_-}\right)}{10}\right)
+\frac{p(p+1)}{2}\ln\left(5p\sqrt{\kappa^2 \cosh\left(\frac{2\kappa}{5}\right)+\frac{1}{2}}\right)
\right.\right.\\
&\left.\left.
+\frac{p(p-1)}{2}\ln\left(\frac{10\lambda_+}{\lambda_-}\right)
+(p-1)\ln\left(\frac{4}{5}+\frac{52\lambda_+}{5\lambda_-}\ln\left(\frac{\lambda_+}{\lambda_-}\right)\right) 
\right] \right).
\end{align*}
\end{coro}

Just like in the general case, we define $\mathcal{C}_1$ by:
\begin{align*}
\mathcal{C}_1&=C_{\Upsilon}
+\frac{1}{2} \ln\left(\frac{25 p \lambda_+ \left(\kappa^2 \cosh \left(\frac{2\kappa}{5}\right)+\frac{1}{2}\right)}{\gamma_\kappa L_- \lambda_-^2}\right)+\ln\left(5p\sqrt{\kappa^2 \cosh\left(\frac{2\kappa}{5}\right)+\frac{1}{2}}\right)\\
&+\frac{2}{p(p+1)}
\left(\ln(c_U)
+\ln\left(\frac{4+129\ln\left(\frac{L_+}{L_-}\right)}{10}\right)+(p-1)\ln\left(\frac{4}{5}+\frac{52\lambda_+}{5\lambda_-}\ln\left(\frac{\lambda_+}{\lambda_-}\right)\right)\right)
\\
&+\frac{p-1}{p+1}\ln\left(\frac{10\lambda_+}{\lambda_-}\right)
\end{align*}
and remind that $\mathfrak{C}=C_W+\ln\left(\frac{20\sqrt{K_{\max}-1}}{3\sqrt{3}}\right)+\mathcal{C}_1$ is an upper-bound for $C_m$.
We recall that $C_W=\ln\left(\sqrt{2}+\maxaW\binom{\degW+d}{d}\right)$ and $C_\Upsilon=\ln\left(\sqrt{2}+\sqrt{p}\maxaUpsilon\binom{\degUpsilon+d}{d}\right)$, and observe that $\mathfrak{C}$ does not depend on the model $S_m$ in the collection since $\mathfrak{C}$ only depends on $K_{\max},\maxaW,\degW,\maxaUpsilon, \degUpsilon, p, d, \kappa $ and the parameters defining $V_K$.
Then we can apply the result in the general case to the collection $(S_m)$ in which each model is defined by a number of components $K$, a common free structure on the covariance matrix $K$-tuple and a common maximal degree for the sets $W_K$ and $\Upsilon_K$.
$(x_m=K)_{m\in \mathcal{M}}$ satisfies Kraft inequality, since
$\sum_{m\in\mathcal{M}} e^{-x_m}\leq \frac{1}{e-1}$.
We obtain an oracle inequality with $\pen(m)=\kappa \left((C+ \ln(n)) \dim(S_m)+x_m \right)$, where $C=2(\sqrt{\mathfrak{C}}+\sqrt{\pi})^2$,  $\dim(S_m)=(K-1+Kp) \binom{\degW+d}{d}+Kp\frac{p+1}{2}$ and $x_m=K$ for the selection of the number of components in the mixture. 
If we change the structure $V_K$ over the covariance matrices, it only changes the constant $\Xi$ in Kraft inequality, since there a finite number of possible structures for a fixed $K$ and the sum $\sum_{m\in\mathcal{M}} e^{-x_m}$ can be rewritten $\sum_{K \in \mathbb{N}^*} \sum_{m \in \mathcal{M}|m(1)=K} e^{-x_m}$.

\section{Lemma Proofs}
\label{appendiceC}

In this section, we provide the proofs of the main lemmas used in the first appendix, to prove Theorem~\ref{theo:selectdim}.
It begins with bracketing entropy's decomposition, then we focus on the bracketing entropy of the weight's families in the general case and in our example, followed by the analysis of the bracketing entropy of Gaussian families.
\ifbool{Preprint}{}{For sake of brevity, some technical proofs are
  omitted  here and they can be found in an extended version.}

\subsection{Bracketing entropy's decomposition}
\label{decomp section}

\begin{lem}
\label{decomp lemma}
Let \begin{align*}
\mathcal{P}&=\left\{\pi=(\pi_k)_{1 \leq k \leq K}/\forall k, \pi_k:\mathcal{X} \rightarrow \mathbb{R}^+ \mbox{ and } \forall x\in \mathcal{X}, \sum_{k=1}^K \pi_k(x)=1\right\},\\
\Psi&=\left\{(\psi_1,\ldots,\psi_K)/\forall k, \psi_k:\mathcal{X} \times \mathcal{Y} \rightarrow \mathbb{R}^+,\mbox{ and }\forall x, \forall k, \int \psi_k(x,y) dy=1 \right\},\\
\mathcal{C}&=\left\{(x,y) \mapsto \sum_{k=1}^K \pi_k(x) \psi_k(x,y)/\pi \in \mathcal{P}, \psi \in \Psi \right\}.
\end{align*}
Then for all $\delta$ in $[0;\sqrt{2}]$,
\begin{align*}
H_{[.],\underset{x} \sup \, d_{y}}\left(\delta,\mathcal{C}\right) \leq H_{[.],\underset{x} \sup  \,  d_k}\left(\frac{\delta}{5},\mathcal{P}\right) + H_{[.],\underset{x} \sup  \, è\underset{k} \max \,  d_y}\left(\frac{\delta}{5},\Psi\right).
\end{align*}
\end{lem}

The proof mimics the one of Lemma 7 from \cite{ConditionalDensity}.

\begin{proof}
First we will exhibit a covering of bracket of $\mathcal{C}$.

Let $([\pi^{i,-}, \pi^{i,+}])_{1\leq i\leq N_{\mathcal{P}}}$ be a minimal
    covering of $\delta$
    bracket for $\underset{x} \sup  \,  d_k$ of $\mathcal{P}$:
\[
   \forall i \in \{1,\ldots,N_{\mathcal{P}}\},
    \forall x \in \mathcal{X}, 
     d_k(\pi^{i,-}(x) ,\pi^{i,+}(x))\leq \delta.
\]
Let $([\psi^{i,-},\psi^{i,+}])_{1\leq i \leq N_{\Psi}}$ be a minimal covering of $\delta$ bracket for $\underset{x} \sup  \, \underset{k} \max \,d_y$ of $\Psi$:
 $\forall i \in \{1,\ldots,N_{\Psi}\},
 \forall x \in \mathcal{X},
\forall k \in \{1,\ldots,K\},$
 $d_y(\psi_k^{i,-}(x,.) ,\psi_k^{i,+}(x,.))\leq \delta$.
Let $s$ be a density in $\mathcal{C}$. By definition, there is $\pi$ in $\mathcal{P}$ and $\psi$ in $\Psi$
such that for all $(x,y)$ in $\mathcal{X} \times \mathcal{Y},
s(y|x)=\sum_{k=1}^K \pi_k(x) \psi_k(x,y)$.
 
Due to the covering, there is $i$ in $\{1,\ldots,N_{\mathcal{P}}\}$
such that 
\[
\forall x \in \mathcal{X},
\forall k \in \{1,\ldots,K\},
\pi_k^{i,-}(x) \leq \pi_k(x) \leq \pi_k^{i,+}(x).
\]
There is also $j$ in $\{1,\ldots,N_{\Psi}\}$ such that 
\[
\forall x \in \mathcal{X},
\forall k \in \{1,\ldots,K\},
\forall y \in \mathcal{Y},
\psi_k^{j,-}(x,y) \leq \psi_k(x,y) \leq \psi_k^{j,+}(x,y).
\]

Since for all $x$, for all $k$ and for all $y$, $\pi_k(x)$ and $\psi_k(x,y)$ are non-negatives, we may multiply term-by-term and sum these inequalities over $k$ to obtain:
\[
\forall x \in \mathcal{X},
\forall y \in \mathcal{Y},
\sum_{k=1}^K \left(\pi_k^{i,-}(x)\right)_+
\left(\psi_k^{j,-}(x,y)\right)_+ \leq s(y|x) \leq \sum_{k=1}^K
\pi_k^{i,+}(x) \psi_k^{j,+}(x,y).
\]
$\displaystyle 
 \left(\left[\sum_{k=1}^K \left(\pi_k^{i,-}\right)_+
    \left(\psi_k^{j,-}\right)_+, \sum_{k=1}^K \pi_k^{i,+}
    \psi_k^{j,+}\right]\right)_{\substack{1\leq i\leq
    N_{\mathcal{P}}\\1\leq j \leq N_{\Psi}}}
$
 is thus a bracket covering of
$\mathcal{C}$.

Now, we focus on brackets' size using lemmas from
\cite{ConditionalDensity} (namely Lemma 11, 12, 13), 
To lighten the notations, $\pi_k^-$ and $\psi_k^-$ are supposed non-negatives for all $k$.
Following their Lemma 12, only using Cauchy-Schwarz inequality, we prove that
\begin{multline*}
\underset{x} \sup \, d_y^2\left(\sum_{k=1}^K \pi_k^-(x) \psi_k^-(x,.),\sum_{k=1}^K \pi_k^+(x) \psi_k^+(x,.)\right)\\
\leq \underset{x} \sup \, d_{y,k}^2(\pi^-(x) \psi^-(x,.), \pi^+(x) \psi^+(x,.))
\end{multline*}
Then, using Cauchy-Schwarz inequality again, we get by their Lemma 11:
\begin{multline*}
 \sup_x d_{y,k}^2(\pi^-(x) \psi^-(x,.), \pi^+(x) \psi^+(x,.))\\
\leq \sup_x \left( \max_k d_y(\psi_k^+(x,.),\psi_k^-(x,.)) \sqrt{\sum_{k=1}^K \pi_k^+(x)}\right.\\
 \left.+d_k(\pi^+(x),\pi^-(x)) \max_k \sqrt{\int \psi_k^-(x,y) dy} \right)^2
\end{multline*}
According to their Lemma 13, $\forall x, \sum_{k=1}^K \pi_k^+(x)\leq 1+2(\sqrt{2}+\sqrt{3})\delta$.
\begin{align*}
&\sup_x \left( \max_k d_y(\psi_k^+(x,.),\psi_k^-(x,.)) \sqrt{\sum_{k=1}^K \pi_k^+(x)}\right. \\
&\qquad \qquad \qquad \qquad \qquad \ \left.+d_k(\pi^+(x),\pi^-(x)) \max_k \sqrt{\int \psi_k^-(x,y) dy} \right)^2\\
& \leq \left( \sqrt{1+2(\sqrt{2}+\sqrt{3})\delta}+ 1 \right)^2 {\delta}^2\\
& \leq (5\delta)^2
\end{align*}
The result follows from the fact we exhibited a $5\delta$ covering of brackets of $\mathcal{C}$, with cardinality $ N_{\mathcal{P}}N_{\Psi}$.
\end{proof}

\subsection{Bracketing entropy of weight's families}
\label{weight}

\subsubsection{When $W_K$ is a compact}
We demonstrate that for any $\delta \in [0;\sqrt{2}]$,
\begin{align*}
H_{[.],\underset{x}\sup \, d_k}\left(\frac{\delta}{5},\mathcal{P}\right)&\leq
H_{\underset{k}\max \, \|\|_{\infty}}\left(\frac{3\sqrt{3}\delta}{20\sqrt{K}},W_K \right)
\end{align*} 

\begin{proof}
We show that
$\forall (w,z) \in (W_K)^2, \forall k \in \{1,\ldots,K\}, \forall x \in \mathcal{X},
 |\sqrt{\pi_{w,k}(x)}-\sqrt{\pi_{z,k}(x)}| \leq F(k,x) d(w,z)$, with $F$ a function and $d$ some distance.
We define $\forall k, \forall u \in \mathbb{R}^K,
A_k(u)=\frac{\exp(u_k)}{\sum_{k=1}^K \exp(u_k)}$, so
$\pi_{w,k}(x)=A_k(w(x))$.

$\forall (u,v) \in (\mathbb{R}^K)^2$,
\begin{align*}
\left|\sqrt{A_k(v)}-\sqrt{A_k(u)}\right|&=\left|\int_0^1 \nabla\left(\sqrt{A_k}\right)(u+t(v-u)).(v-u) dt \right|
\end{align*}
Besides,
\begin{align*}
\nabla\left(\sqrt{A_k}\right)(u)
&=\left(\frac{1}{2}\sqrt{A_k(u)} \frac{\partial}{\partial u_l}(\ln(A_k(u)))\right)_{1\leq l \leq K}\\
&=\left(\frac{1}{2}\sqrt{A_k(u)} \left(\delta_{k,l}-A_l(u)\right)\right)_{1\leq l \leq K}
\end{align*}
\begin{align*}
\left|\sqrt{A_k(v)}-\sqrt{A_k(u)}\right| \mspace{-75mu}\\
&=\frac{1}{2}\left|\int_0^1 \sqrt{A_k(u+t(v-u))} \sum_{l=1}^K \left(\delta_{k,l}-A_l(u+t(v-u))\right)(v_l-u_l) dt \right|\\
&\leq \frac{1}{2} \int_0^1 \sqrt{A_k(u+t(v-u))} \sum_{l=1}^K \left|\delta_{k,l}-A_l(u+t(v-u))\right| \left|(v_l-u_l)\right| dt\\
&\leq \frac{\|v-u\|_{\infty}}{2} \int_0^1 \sqrt{A_k(u+t(v-u))} \sum_{l=1}^K \left|\delta_{k,l}-A_l(u+t(v-u))\right| dt 
\end{align*}
Since $\forall u \in \mathbb{R}^K, \sum_{k=1}^K A_k(u)=1$, $\sum_{l=1}^K |\delta_{k,l}-A_l(u)|=2(1-A_k(u))$
\begin{align*}
\left|\sqrt{A_k(v)}-\sqrt{A_k(u)}\right| &\leq \|v-u\|_{\infty} \int_0^1 \sqrt{A_k(u+t(v-u))}\left(1-A_k(u+t(v-u))\right) dt\\
&\leq \frac{2}{3\sqrt{3}}\|v-u\|_{\infty}
\end{align*}
since $x \mapsto \sqrt{x}(1-x)$ is maximal over [0;1] for $x=\frac{1}{3}$.
We deduce that for any $(w,z)$ in $(W_K)^2$, for all $k$ in $\{1,\ldots,K\}$, for any $x$ in $\mathcal{X}$,
 $|\sqrt{\pi_{w,k}(x)}-\sqrt{\pi_{z,k}(x)}|\leq \frac{2}{3\sqrt{3}} \max_l \|w_l-z_l\|_{\infty}$.
 
By hypothesis, for any positive $\epsilon$, an $\epsilon$-net $\mathcal{N}$ of $W_K$ may be exhibited. Let $w$ be an element of $W_K$. There is a $z$ belonging to the $\epsilon$-net $\mathcal{N}$ such that $\max_l \|z_l-w_l\|_{\infty}\leq \epsilon$.
Since for all $k$ in $\{1,\ldots,K\}$, for any $x$ in $\mathcal{X}$,
\begin{align*}
|\sqrt{\pi_{w,k}(x)}-\sqrt{\pi_{z,k}(x)}|\leq \frac{2}{3\sqrt{3}} \max_l \|w_l-z_l\|_{\infty}\leq \frac{2}{3\sqrt{3}} \epsilon,
\end{align*}
and
\begin{align*}
\sum_{k=1}^K \left(\sqrt{\pi_{z,k}(x)}+\frac{2}{3\sqrt{3}} \epsilon-\sqrt{\pi_{z,k}(x)}+\frac{2}{3\sqrt{3}} \epsilon \right)^2=K\left(\frac{4\epsilon}{3\sqrt{3}}\right)^2,
\end{align*} 
 $\left(\left[\left(\sqrt{\pi_{z}}-\frac{2}{3\sqrt{3}} \epsilon \right)^2;\left( \sqrt{\pi_{z}}+\frac{2}{3\sqrt{3}} \epsilon\right)^2 \right]\right)_{z \in \mathcal{N}}$ is a $\frac{4\epsilon \sqrt{K}}{3\sqrt{3}}$-bracketing cover of $\mathcal{P}$.
As a result, $H_{[],\sup_x d_k}\left(\frac{\delta}{5},\mathcal{P}\right)\leq H_{\max_k \|\|_{\infty}}\left(\frac{3\sqrt{3}}{20\sqrt{K}}\delta,W_K\right)$. 
\end{proof}

\subsubsection{When $W_K=\{0\}\otimes W^{K-1}$ with $W$ a set of polynomials}
We remind that 
\begin{align*}
W=\left\{w:[0;1]^d \rightarrow \mathbb{R}/ w(x)=\sum_{|r|=0}^{\degW} \alpha_r x^r \mbox{ and } \|\alpha\|_{\infty}\leq \maxaW\right\}
\end{align*}
\begin{prop}
For all $\delta \in [0;\sqrt{2}]$,
\begin{align*}
H_{[.],\underset{x}\sup \, d_k}\left(\frac{\delta}{5},\mathcal{P}\right)&\leq
(K-1) \binom{\degW+d}{d}
\\&\qquad\quad\times
\left(\ln\left(\sqrt{2}+\frac{20}{3\sqrt{3}}\maxaW\sqrt{K-1}\binom{\degW+d}{d}\right)
+\ln\left(\frac{1}{\delta}\right) \right).
\end{align*} 
\end{prop}

\begin{proof}
$W_K$ is a finite dimensional compact set. Thanks to the result in the general case, we get
\begin{align*}
H_{[.],\underset{x}\sup \, d_k}\left(\frac{\delta}{5},\mathcal{P}\right)&\leq
H_{\underset{k}\max \, \|\|_{\infty}}\left(\frac{3\sqrt{3}\delta}{20\sqrt{K-1}},W_K \right)\\
&\leq H_{\|.\|_{\infty}}\left(\frac{3\sqrt{3}\delta}{20\sqrt{K-1}\binom{\degW+d}{d}},\left\{\alpha \in \mathbb{R}^{(K-1)\binom{\degW+d}{d}}/\|\alpha\|_{\infty}\leq \maxaW\right\}\right)\\
&\leq (K-1)\binom{\degW+d}{d} \ln \left(1+\frac{20\sqrt{K-1}\maxaW\binom{\degW+d}{d}}{3\sqrt{3}\delta} \right)\\
&\leq (K-1)\binom{\degW+d}{d}
\\&\qquad\quad\times
\left[\ln\left(\sqrt{2}+\frac{20}{3\sqrt{3}}\maxaW\sqrt{K-1}\binom{\degW+d}{d}\right)
+\ln\left(\frac{1}{\delta}\right) \right]
 \end{align*}
The second inequality comes from: for all $w,v$ in $W_K$,\\ $\max_k \|w_k-v_k\|_\infty \leq \max_k \sum_{|r|=0}^{\degW} |\alpha_{k,r}-\beta_{k,r}|
\leq \binom{\degW+d}{d}\max_{k,r}|\alpha_{k,r}-\beta_{k,r}|$.
 \end{proof}

\subsection{Bracketing entropy of Gaussian families}
\label{gauss fam}
\subsubsection{General case}
We rely on a general construction of Gaussian brackets:
\begin{prop}
\label{gauss}
Let $\kappa\geq\frac{17}{29}$,
\( \displaystyle
  \gamma_\kappa  =
  \frac{25(\kappa-\frac{1}{2})}{49(1+\frac{2\kappa}{5})}\).
For any $0<\delta\leq\sqrt{2}$, any $p\geq 1$ 
 and any $\delta_\Sigma \leq
\frac{1}{5\sqrt{\kappa^2 \cosh(\frac{2\kappa}{5}) +
    \frac{1}{2}}}\frac{\delta}{p}$, 

let $(\upsilon,L,A,D) \in \Upsilon \times  [L_-,L_+] \times 
\mathcal{A}(\lambda_-,\lambda_+) \times SO(p)$ and
$(\tilde{\upsilon}, \tilde{L},\tilde{A},\tilde{D})\in \Upsilon \times [L_-,L_+] \times 
\mathcal{A}(\lambda_-,+\infty) \times SO(p)$, define
$\Sigma=L D A D'$ and  $\tilde{\Sigma}=\tilde{L} \tilde{D}
\tilde{A} \tilde{D}'$,
\begin{align*}
  t^-(x,y) = (1+\kappa\delta_\Sigma)^{-p} \Phi_{\tilde{\upsilon}(x),(1+\delta_\Sigma)^{-1} \tilde{\Sigma}}(y)
\quad\text{and}\quad
  t^+(x,y) = (1+\kappa\delta_\Sigma)^{p} \Phi_{\tilde{\upsilon}(x),(1+\delta_\Sigma) \tilde{\Sigma}}(y).
\end{align*}
If
\begin{align*}
\begin{cases}
\forall x \in \mathcal{X}, \|\upsilon(x)-\tilde{\upsilon}(x)\|^2 \leq p \gamma_\kappa L_- \lambda_-
  \frac{\lambda_-}{\lambda_+} \delta_\Sigma^2\\
\left(1+\frac{2}{25}\delta_\Sigma \right)^{-1} \tilde{L} \leq L \leq \tilde{L}\\
\forall 1\leq i \leq p,|A_{i,i}^{-1}-\tilde{A}_{i,i}^{-1}| \leq \frac{1}{10}
\frac{\delta_\Sigma}{\lambda_+}\\
 \forall y \in \mathbb{R}^p,\|Dy-\tilde{D}y\| \leq
\frac{1}{10} \frac{\lambda_-}{\lambda_+}
\delta_\Sigma \|y\|
\end{cases}
\end{align*}
then $[t^-,t^+]$ is a $\delta/5$ Hellinger bracket such that
\(\displaystyle
t^-(x,y) \leq \Phi_{\upsilon(x),\Sigma}(y) \leq t^+(x,y).
\)
\end{prop}

Admitting this proposition, we are brought to construct nets over the spaces of the means, the volumes, the eigenvector matrices and the normalized eigenvalue matrices. We consider three cases:
the parameter (mean, volume, matrix) is known ($\star=0$), unknown but common to all classes ($\star=c$), unknown and possibly different for every class ($\star=K$). For example, $[\nu_K, L_0, D_c, A_0]$ denotes a model in which only means are free and eigenvector matrices are assumed to be equal and unknown.

If the means are free ($\star=K$), we construct a grid $G_{\Upsilon_K}$ over $\Upsilon_K$, which is compact. Since
\begin{align*}
&H_{\max_k \sup_x \|.\|_2}\left(\sqrt{p \gamma_\kappa L_- \lambda_-
  \frac{\lambda_-}{\lambda_+}} \delta_\Sigma,\Upsilon_K\right) \leq
D_{\Upsilon_K} \left(C_{\Upsilon}+\ln \left(\frac{1}{\sqrt{p \gamma_\kappa L_- \lambda_-
  \frac{\lambda_-}{\lambda_+}} \delta_\Sigma} \right) \right),\\
 &\left|G_{\Upsilon_K}\left(\sqrt{p \gamma_\kappa L_- \lambda_-
  \frac{\lambda_-}{\lambda_+}} \delta_\Sigma\right) \right|\leq \left(C_{\Upsilon}+\ln \left(\frac{1}{\sqrt{p \gamma_\kappa L_- \lambda_-
  \frac{\lambda_-}{\lambda_+}} \delta_\Sigma} \right) \right)^{D_{\Upsilon_K}}.
\end{align*}
If the means are common and unknown ($\star=c$), belonging to $\Upsilon_1$ , we construct a grid $G_{\Upsilon_c}\left(\sqrt{p \gamma_\kappa L_- \lambda_-
  \frac{\lambda_-}{\lambda_+}} \delta_\Sigma\right)$ over $\Upsilon_1$
with cardinality at most 
\[
\left(C_{\Upsilon}+\ln \left(\frac{1}{\sqrt{p \gamma_\kappa L_- \lambda_-
  \frac{\lambda_-}{\lambda_+}} \delta_\Sigma} \right)
\right)^{D_{\Upsilon_1}}.
\]
Finally, if the means are known ($\star=0$), we  do not need to construct a grid.
In the end, $\left|G_{\Upsilon_\star}\left(\sqrt{p \gamma_\kappa L_- \lambda_-
  \frac{\lambda_-}{\lambda_+}} \delta_\Sigma\right)\right|\leq \left(C_{\Upsilon}+\ln \left(\frac{1}{\sqrt{p \gamma_\kappa L_- \lambda_-
  \frac{\lambda_-}{\lambda_+}} \delta_\Sigma} \right) \right)^{Z_{\upsilon,\star}}$, with $Z_{\upsilon,K}=D_{\Upsilon_K}$, $Z_{\upsilon,c}=D_{\Upsilon_1}$ and $Z_{\upsilon,0}=0$.
  
Then, we consider the grid $G_{L}$ over $[L_-,L_+]$:
\begin{align*}
G_L\left(\frac{2}{25}\delta_\Sigma \right)&=\left\{L_- \left(1+\frac{2}{25}\delta_\Sigma\right)^g/g\in \mathbb{N},L_- \left(1+\frac{2}{25}\delta_\Sigma\right)^g\leq L_+\right\}\\
\left|G_L\left(\frac{2}{25}\delta_\Sigma \right) \right|&\leq
1+\frac{\ln \left(\frac{L_+}{L_-} \right)}{\ln\left(1+\frac{2}{25}\delta_\Sigma \right)}
\end{align*}
Since $\delta_\Sigma \leq \frac{2}{5}$, $\ln\left(1+\frac{2}{25}\delta_\Sigma \right)\geq \frac{10}{129}\delta_\Sigma$.
\begin{align*}
\left|G_L\left(\frac{2}{25}\delta_\Sigma \right) \right|&\leq
1+\frac{129\ln \left(\frac{L_+}{L_-} \right)}{10\delta_\Sigma}
\leq \frac{4+129\ln \left(\frac{L_+}{L_-} \right)}{10\delta_\Sigma}
\end{align*}

By definition of a net, for any $D \in SO(p)$ there is a $\tilde{D}\in G_D\left(\frac{1}{10} \frac{\lambda_-}{\lambda_+}
\delta_\Sigma\right)$ such that $ \forall y \in \mathbb{R}^p,\|Dy-\tilde{D}y\| \leq
\frac{1}{10} \frac{\lambda_-}{\lambda_+}
\delta_\Sigma \|y\|$.
There exists a universal constant $c_U$ such that $\left|G_D\left(\frac{1}{10} \frac{\lambda_-}{\lambda_+}\delta_\Sigma\right) \right| \leq
c_U \left(\frac{10\lambda_+}{\lambda_- \delta_\Sigma}\right)^{\frac{p(p-1)}{2}}$.

For the grid $G_A$, we look at the condition on the $p-1$ first diagonal values and obtain:
\begin{align*}
\left|G_A\left(\frac{1}{10} \frac{\lambda_-}{\lambda_+}\delta_\Sigma\right) \right|\leq 
\left(2+\frac{\ln\left(\frac{\lambda_+}{\lambda_-}\right)}{\ln\left(1+\frac{1}{10} \frac{\lambda_-}{\lambda_+}\delta_\Sigma\right)} \right)^{p-1}
\end{align*}
Since $\delta_\Sigma \leq \frac{2}{5}$, $\ln\left(1+\frac{1}{10} \frac{\lambda_-}{\lambda_+}\delta_\Sigma\right)\geq \frac{5}{52}\frac{\lambda_-}{\lambda_+}\delta_\Sigma$, then
\begin{align*}
\left|G_A\left(\frac{1}{10} \frac{\lambda_-}{\lambda_+}\delta_\Sigma\right) \right|\leq 
\left(2+\frac{52}{5\delta_\Sigma}\frac{\lambda_+}{\lambda_-}\ln\left(\frac{\lambda_+}{\lambda_-}\right) \right)^{p-1}
\leq \left(4+52\frac{\lambda_+}{\lambda_-}\ln\left(\frac{\lambda_+}{\lambda_-}\right) \right)^{p-1}
\left(\frac{1}{5\delta_\Sigma}\right)^{p-1}
\end{align*}

Let $Z_{L,0}=Z_{D,0}=Z_{A,0}=0$, $Z_{L,c}=Z_{D,c}=Z_{A,c}=1$, $Z_{L,K}=Z_{D,K}=Z_{A,K}=K$.
We define $f_{\upsilon,\star}$ from $\Upsilon_\star$ to $\Upsilon_K$ by
$\begin{cases}
0\mapsto (\upsilon_{0,1},\ldots,\upsilon_{0,1}) \mbox{ if }\star=0\\
\upsilon \mapsto (\upsilon,\ldots, \upsilon) \mbox{ if }\star=c\\
(\upsilon_1,\ldots,\upsilon_K) \mapsto (\upsilon_1,\ldots,\upsilon_K) \mbox{ if } \star=K
\end{cases}$
and similarly $f_{L,\star}$, $f_{D,\star}$ and $f_{A,\star}$, respectively from $\left(\mathbb{R}_+ \right)^{Z_{L,\star}}$ into $\left(\mathbb{R}_+ \right)^K$, from $\left(SO(p)\right)^{Z_{D,\star}}$ into $\left(SO(p)\right)^K$ and from $\mathcal{A}(\lambda_-,\lambda_+)^{Z_{A,\star}}$
into $\mathcal{A}(\lambda_-,\lambda_+)^K$.

We define 
\[
\Gamma:(\upsilon_1,\ldots,\upsilon_K,L_1,\ldots,L_K,D_1,\ldots,D_K,A_1,\ldots,A_K)
\mapsto (\upsilon_k,L_kD_kA_kD_k')_{1\leq k \leq K}
\]
 and $\Psi:(\upsilon_k,\Sigma_k)_{1\leq k \leq K}
 \mapsto (\Phi_{\upsilon_k,\Sigma_k})_{1\leq k \leq K}$.
 The image of $\Upsilon_\star \times [L_-,L_+]^{Z_{L,\star}} \times SO(p)^{Z_{D,\star}} \times \mathcal{A}(\lambda_-,\lambda_+)^{Z_{A,\star}}$ by $\Psi \circ \Gamma \circ (f_{\upsilon,\star} \otimes f_{L,\star}\otimes f_{D,\star}\otimes f_{A,\star})$ is the set $\mathcal{G}$ of all K-tuples of Gaussian densities of type $[\upsilon_\star,L_\star,D_\star,A_\star]$.

Now, we define $B$:
\[
(\upsilon_k,\Sigma_k)_{1\leq k \leq K} \mapsto 
\left((1+\kappa\delta_\Sigma)^{-p} \Phi_{\upsilon_k,(1+\delta_\Sigma)^{-1}\Sigma_k},
(1+\kappa\delta_\Sigma)^p \Phi_{\upsilon_k,(1+\delta_\Sigma)\Sigma_k}
\right)_{1\leq k \leq K}.
\]
 The image of $G_{\Upsilon_\star} \times G_L^{Z_{L,\star}} \times G_D^{Z_{D,\star}} \times G_A^{Z_{A,\star}}$ by $B\circ \Gamma \circ (f_{\upsilon,\star} \otimes f_{L,\star}\otimes f_{D,\star}\otimes f_{A,\star})$ is a $\delta/5$-bracket covering of $\mathcal{G}$, with cardinality bounded by
\begin{align*}
&\left(\frac{\sqrt{\lambda_+} \exp\left(C_{\Upsilon} \right)}{\sqrt{p \gamma_\kappa L_- \lambda_-^2}\delta_\Sigma}\right)^{Z_{\Upsilon,\star}}\times \left(
\frac{4+129\ln \left(\frac{L_+}{L_-} \right)}{10\delta_\Sigma} \right)^{Z_{L,\star}}
 \times
c_U^{Z_{D,\star}} \left(\frac{10\lambda_+}{\lambda_-\delta_\Sigma}
\right)^{\frac{p(p-1)}{2}Z_{D,\star}}\\ 
&\times
 \left(4+52\frac{\lambda_+}{\lambda_-}\ln\left(\frac{\lambda_+}{\lambda_-}\right) \right)^{(p-1) Z_{A,\star}}
\left(\frac{1}{5\delta_\Sigma}\right)^{(p-1) Z_{A,\star}}
\end{align*}
 Taking $\delta_\Sigma=\frac{1}{5\sqrt{\kappa^2 \cosh \left(\frac{2\kappa}{5}\right)+\frac{1}{2}}} \frac{\delta}{p}$, we obtain
\begin{align*}
H_{[.],\sup_x \max_k d_y}\left(\frac{\delta}{5},\mathcal{G}\right)
\leq \mathcal{D} \left(\mathcal{C}+\ln \left(\frac{1}{\delta}\right) \right)
\end{align*}
with $\displaystyle\mathcal{D}=Z_{\upsilon,\star}+Z_{L,\star}+\frac{p(p-1)}{2}Z_{D,\star}+(p-1)Z_{A,\star}$ and 
\begin{align*}
\mathcal{C}&=\ln\left(5p\sqrt{\kappa^2 \cosh \left(\frac{2\kappa}{5}\right)+\frac{1}{2}}\right)+\frac{Z_{\upsilon,\star}C_{\Upsilon}}{\mathcal{D}} 
+\frac{Z_{\upsilon,\star}}{2\mathcal{D}} \ln\left(\frac{\lambda_+}{p\gamma_\kappa L_- \lambda_-^2}\right)\\
&
+\frac{Z_{L,\star}}{\mathcal{D}} \ln\left(\frac{4+129\ln\left(\frac{L_+}{L_-}\right)}{10}\right)
+\frac{Z_{D,\star}}{\mathcal{D}} \left( \ln(c_U)+\frac{p(p-1)}{2} \ln\left(\frac{10\lambda_+}{\lambda_-} \right)\right)\\
&+\frac{Z_{A,\star}(p-1)}{\mathcal{D}}
\ln\left(\frac{4}{5}+\frac{52\lambda_+}{5\lambda_-}\ln\left(\frac{\lambda_+}{\lambda_-}\right)\right)
\end{align*}

\subsubsection{With polynomial means}

Using previous work, we only have to handle $\Upsilon_K$'s bracketing entropy.
Just like for $W_K$, we aim at bounding the bracketing entropy by the entropy of the parameters' space.

We focus on the example where $\Upsilon_K=\Upsilon^K$ and
\[
\Upsilon=\left\lbrace
\upsilon:[0;1]^d \rightarrow \mathbb{R}^p \Big| \forall j \in \{1,\ldots,p\}, \forall x, \upsilon_j(x)=\sum_{|r|=0}^{\degUpsilon} \alpha_r^{(j)} x^r \mbox{ and } \|\alpha\|_\infty\leq \maxaUpsilon
\right\rbrace \]
We consider for any $\upsilon$, $\nu$ in $\Upsilon$ and any $x$ in $[0;1]^d$,
\begin{align*}
\|\upsilon(x)-\nu(x)\|_2^2 &=\sum_{j=1}^p \left(\sum_{|r|=0}^{\degUpsilon} \left(\alpha_r^{(j)}-\beta_r^{(j)} \right) x^r \right)^2\\
&\leq \sum_{j=1}^p \left(\sum_{|r|=0}^{\degUpsilon} \left(\alpha_r^{(j)}-\beta_r^{(j)} \right)^2 \right) \left(\sum_{|r|=0}^{\degUpsilon} x^{2r} \right)\\
&\leq \binom{\degUpsilon +d}{d} \sum_{j=1}^p \sum_{|r|=0}^{\degUpsilon} \left(\alpha_r^{(j)}-\beta_r^{(j)} \right)^2 \\
&\leq p  \binom{\degUpsilon +d}{d}^2 \max_{j,r} \left(\alpha_r^{(j)}-\beta_r^{(j)} \right)^2 
\end{align*}
So,
\begin{align*}
H_{[.],\max_k \sup_x \|\|_2}\left(\delta,\Upsilon_K\right)
&\leq H_{\max_{k,j,r}|.|}\left(\frac{\delta}{\sqrt{p}\binom{\degUpsilon +d}{d}}, \left\lbrace \left(\alpha_r^{(j,k)} \right)_{\overset{1\leq j \leq p}{\underset{1\leq k \leq K}{|r|\leq \degUpsilon}}}
 \Big| \|\alpha\|_\infty \leq \maxaUpsilon \right\rbrace \right)\\
&\leq pK\binom{\degUpsilon +d}{d} \ln \left(1+\frac{\sqrt{p}\binom{\degUpsilon +d}{d}\maxaUpsilon}{\delta} \right)\\
&\leq pK\binom{\degUpsilon +d}{d} \left[\ln \left(\sqrt{2}+\sqrt{p}\binom{\degUpsilon +d}{d}\maxaUpsilon\right)+\ln \left(\frac{1}{\delta} \right)\right]\\
&\leq D_{\Upsilon_K} \left( C_\Upsilon+\ln \left(\frac{1}{\delta}\right)\right)
\end{align*}
with $ D_{\Upsilon_K}=pK\binom{\degUpsilon+d}{d}$ and $C_\Upsilon=\ln \left(\sqrt{2}+\sqrt{p}\binom{\degUpsilon +d}{d}\maxaUpsilon\right)$.

\ifbool{Preprint}{

\subsection{Proof of the key proposition to handle bracketing entropy
  of Gaussian families}

\subsubsection{Proof of  Proposition \ref{gauss}}
\begin{proof}
$[t^-,t^+]$ is a $\delta$/5 bracket.\\
Since $(1+\delta_\Sigma)\tilde{\Sigma}^{-1}-(1+\delta_\Sigma)^{-1}\tilde{\Sigma}^{-1}=((1+\delta_\Sigma)-(1+\delta_\Sigma)^{-1})\tilde{\Sigma}^{-1}$ is a positive-definite matrix, Maugis and Michel's lemma can be applied.
\begin{lem} (\cite{Maugis}) Let $\Phi_{\upsilon_1,\Sigma_1}$ and $\Phi_{\upsilon_2,\Sigma_2}$ be two Gaussian densities with full rank covariance matrix in dimension p such that $\Sigma_1^{-1}-\Sigma_2^{-1}$ is a positive definite matrix. For any $y \in \mathbb{R}^p$,
\[
\frac{\Phi_{\upsilon_1,\Sigma_1}(y)}{\Phi_{\upsilon_2,\Sigma_2}(y)} \leq
\sqrt{\frac{|\Sigma_2|}{|\Sigma_1|}}
\exp \left(\frac{1}{2} (\upsilon_1-\upsilon_2)' (\Sigma_2-\Sigma_1)^{-1} (\upsilon_1-\upsilon_2)  \right).
\]
\end{lem}
Thus, $\forall x \in \mathcal{X}, \forall y \in \mathbb{R}^p,$
      \begin{align*}
 \frac{t^-(x,y)}{t^+(x,y)}=&
  \frac{(1+\kappa \delta_\Sigma)^{-p}}{(1+\kappa \delta_\Sigma)^{p}}
      \frac{\Phi_{\upsilon(x),(1+\delta_\Sigma)^{-1}
	  \tilde{\Sigma}}(y)}{\Phi_{\upsilon(x),(1+\delta_\Sigma) \tilde{\Sigma}}(y)} 
  \leq \frac{1}{(1+\kappa\delta_\Sigma)^{2 p}} \sqrt{%
    \frac{(1+\delta_\Sigma)^{p}}{(1+\delta_\Sigma)^{-p}}}\\
  =& \left(\frac{1+\delta_\Sigma}{(1+\kappa\delta_\Sigma)^2}\right)^{p}
  = \left(\frac{1+\delta_\Sigma}{1+ 2 \kappa \delta_\Sigma + \kappa^2
    \delta_\Sigma^2}\right)^{p}
  \leq 1
    \end{align*}
  For all $x$ in $\mathcal{X}$,
    \begin{align*}
      d_y^2(t^-,t^+) & = \int t^-(x,y)\mathrm{d} y  + \int t^+(x,y) \, \mathrm{d} y
    - 2 \int \sqrt{t^-(x,y)} \sqrt{t^+(x,y)}  \mathrm{d} y\\
	& = (1+\kappa\delta_\Sigma)^{-p}+(1+\kappa\delta_\Sigma)^{p}- 2
      (1+\kappa\delta_\Sigma)^{-p/2} (1+\kappa\delta_\Sigma)^{p/2}\\
&\qquad \qquad \times \int \sqrt{\Phi_{\upsilon(x),(1+\delta_\Sigma)^{-1} \tilde{\Sigma}}(y)}
      \sqrt{\Phi_{\upsilon(x),(1+\delta_\Sigma) \tilde{\Sigma}}(y)} \, \mathrm{d} y\\
      & = (1+\kappa\delta_\Sigma)^{-p}+(1+\kappa\delta_\Sigma)^{p}-
      \left(2 \right.\\
&\qquad \qquad \qquad \left. -d_y^2\left(\Phi_{\upsilon(x),(1+\delta_\Sigma)^{-1}
	    \tilde{\Sigma}}(y),\Phi_{\upsilon(x),(1+\delta_\Sigma)
	    \tilde{\Sigma}}(y)\right)\right).
    \end{align*}
Using the following lemma,
\begin{lem}\label{lemma15}
Let $\Phi_{\upsilon_1,\Sigma_1}$ and $\Phi_{\upsilon_2,\Sigma_2}$ be two Gaussian densities with full rank covariance matrix in dimension p, then
\begin{align*}
d^2\left(\Phi_{\upsilon_1,\Sigma_1},\Phi_{\upsilon_2,\Sigma_2}\right)&=
2\left(1-2^{p/2} |\Sigma_1 \Sigma_2|^{-1/4} |\Sigma_1^{-1}+\Sigma_2^{-1}|^{-1/2} \right.\\
&\left.\times\exp \left(-\frac{1}{4} (\upsilon_1-\upsilon_2)' (\Sigma_1+\Sigma_2)^{-1} (\upsilon_1-\upsilon_2) \right) \right).
\end{align*}
\end{lem}
 we obtain
\begin{align*}
 d_y^2(t^-,t^+) & =(1+\kappa\delta_\Sigma)^{-p}+(1+\kappa\delta_\Sigma)^{p}-2\, 2^{p/2} \left( (1+\delta_\Sigma) + (1+\delta_\Sigma)^{-1}
 \right)^{-p/2}\\
&=2 - 2\, 2^{p/2} \left( (1+\delta_\Sigma) + (1+\delta_\Sigma)^{-1}
 \right)^{-p/2}+ (1+\kappa\delta_\Sigma)^{-p}-2 \\
&\qquad \qquad +(1+\kappa\delta_\Sigma)^{p}
\end{align*}
Applying Lemma \ref{lemma16} 
\begin{lem}\label{lemma16}
For any $0<\delta \leq \sqrt{2}$ and any $p\geq 1$, let $\kappa \geq
\frac{1}{2}$ and \linebreak $\delta_\Sigma \leq
\frac{1}{5\sqrt{\kappa^2 \cosh(\frac{2\kappa}{5}) +
    \frac{1}{2}}}\frac{\delta}{p}$,
then \[
\delta_\Sigma \leq \frac{2}{5p} \leq \frac{2}{5}.\]
\end{lem}
and
\begin{lem}\label{lemma17}
For any $p \in \mathbb{N}^*$, for any $\delta_\Sigma >0$,
\[
2-2^{p/2+1} \left( (1+\delta_\Sigma)+(1+\delta_\Sigma)^{-1} \right)^{-p/2}
\leq \frac{p{\delta_\Sigma}^2}{2} \leq \frac{p^2{\delta_\Sigma}^2}{2}
\]

Furthermore, if $p\delta_\Sigma \leq c$, then
\[
(1+\kappa \delta_\Sigma)^p+(1+\kappa \delta_\Sigma)^{-p}-2
\leq \kappa^2 \cosh(\kappa c) p^2 {\delta_\Sigma}^2.
\]
\end{lem}

with $c=\frac{2}{5}$, it comes out that:\[
\underset{x}\sup \, d_y^2(t^-(x,y),t^+(x,y))\leq \left( \frac{\delta}{5} \right)^2.\]

 Now, we show that for all $x$ in $\mathcal{X}$, for all $y$ in $\mathbb{R}^p$, $t^-(x,y) \leq \Phi_{\upsilon(x),\Sigma}(y) \leq t^+(x,y)$.
 We use therefore Lemma \ref{lemma18}, thanks to the hypothesis made on covariance matrices.
 
 \begin{lem}\label{lemma18}
Let $(L,A,D)\in [L_-,L_+] \times \mathcal{A}(\lambda_-,\lambda_+) \times SO(p)$ and $(\tilde{L},\tilde{A},\tilde{D})\in [L_-,L_+] \times \mathcal{A}(\lambda_-,\infty) \times SO(p)$, define $\Sigma=LDAD'$ and $\tilde{\Sigma}=\tilde{L}\tilde{D}\tilde{A}\tilde{D}'$. If

$\begin{cases}
(1+\delta_L)^{-1} \tilde{L} \leq L \leq \tilde{L}\\
\forall 1\leq i\leq p, |A_{i,i}^{-1}-{\tilde{A}_{i,i}}^{-1}|\leq \delta_A \lambda_-^{-1}\\
\forall y\in \mathbb{R}^p,\|Dy-\tilde{D}y\|\leq \delta_D \|y\|
\end{cases}$\\
then $(1+\delta_\Sigma)\tilde{\Sigma}^{-1}-\Sigma^{-1}$ and $\Sigma^{-1}-(1+\delta_\Sigma)^{-1} \tilde{\Sigma}^{-1}$ satisfy
\begin{align*}
\forall y\in \mathbb{R}^p,& y'\left((1+\delta_\Sigma)\tilde{\Sigma}^{-1}-\Sigma^{-1}\right)y \geq
\tilde{L}^{-1} \left( (\delta_\Sigma-\delta_L) \lambda_+^{-1}-(1+\delta_\Sigma)\lambda_-^{-1}(2\delta_D+\delta_A) \right) \|y\|^2\\
\forall y\in \mathbb{R}^p,&
y'\left(\Sigma^{-1}-(1+\delta_\Sigma)^{-1}\tilde{\Sigma}^{-1} \right)y \geq
\frac{\tilde{L}^{-1}}{1+\delta_\Sigma} \left(\delta_\Sigma \lambda_+^{-1}-\lambda_-^{-1}(2\delta_D+\delta_A) \right) \|y\|^2
\end{align*}
\end{lem}
Using
$\begin{cases}
\delta_L=\frac{2}{25}\delta_\Sigma\\
\delta_D=\delta_A=\frac{1}{10}\frac{\lambda_-}{\lambda_+}\delta_\Sigma
\end{cases}$\\
we get lower bounds of the same order:
\begin{align*}
\forall y\in \mathbb{R}^p,& y'\left((1+\delta_\Sigma)\tilde{\Sigma}^{-1}-\Sigma^{-1}\right)y \geq
\frac{\tilde{L}^{-1}}{2\lambda_+}\delta_\Sigma \|y\|^2\\
\forall y\in \mathbb{R}^p,&
y'\left(\Sigma^{-1}-(1+\delta_\Sigma)^{-1}\tilde{\Sigma}^{-1} \right)y \geq
\frac{\tilde{L}^{-1}}{1+\delta_\Sigma} \frac{7}{10\lambda_+}\delta_\Sigma \|y\|^2
\end{align*}
 
Let's compare $\Phi_{\upsilon,\Sigma}$ and $t^+$.
\begin{align*}
  &\frac{\Phi_{\upsilon(x),\Sigma}(y)}{(1+\kappa \delta_\Sigma)^{p}
    \Phi_{\tilde{\upsilon}(x),(1+\delta_\Sigma) \tilde{\Sigma}}(y)}\\
& \leq (1+\kappa \delta_\Sigma)^{-p} \left( \sqrt{\frac{|(1+\delta_\Sigma) \tilde{\Sigma}|}{|\Sigma|}}
\exp \left( 
 \frac{1}{2} (\upsilon(x)-\tilde{\upsilon}(x))' \left( (1+\delta_\Sigma) \tilde{\Sigma} - \Sigma \right)^{-1} (\upsilon(x)-\tilde{\upsilon}(x))
 \right)
 \right)\\
& \leq \frac{(1+\delta_\Sigma)^{p/2}}{(1+\kappa \delta_\Sigma)^{p}} \left( \sqrt{\frac{|
      \tilde{\Sigma}|}{|\Sigma|}} \exp \left( 
 \frac{1}{2} (\upsilon(x)-\tilde{\upsilon}(x))' \left( (1+\delta_\Sigma) \tilde{\Sigma} - \Sigma \right)^{-1} (\upsilon(x)-\tilde{\upsilon}(x))
 \right)
 \right).
\end{align*}
But, \begin{align*}
     \left( (1+\delta_\Sigma) \tilde{\Sigma} - \Sigma \right)^{-1}
&= \left( (1+\delta_\Sigma) \tilde{\Sigma} (\Sigma^{-1}-(1+\delta_\Sigma)^{-1} \tilde{\Sigma}^{-1}) \Sigma \right)^{-1}\\
&=(1+\delta_\Sigma)^{-1} \Sigma ^{-1} (\Sigma^{-1}-(1+\delta_\Sigma)^{-1} \tilde{\Sigma}^{-1})^{-1} \tilde{\Sigma}^{-1}
    \end{align*}
Thus by Lemma \ref{lemma18},
\begin{align*}
 &(\upsilon(x)-\tilde{\upsilon}(x))' \left( (1+\delta_\Sigma) \tilde{\Sigma} - \Sigma \right)^{-1} (\upsilon(x)-\tilde{\upsilon}(x))\\
&\leq (1+\delta_\Sigma)^{-1} L_-^{-1} \lambda_-^{-1}
(1+\delta_\Sigma) \tilde{L} \frac{10}{7}\lambda_+ \delta_\Sigma^{-1}
 \tilde{L}^{-1} \lambda_-^{-1} \|\upsilon(x)-\tilde{\upsilon}(x)\|^2\\
&\leq \frac{10}{7} L_-^{-1} \lambda_-^{-2} \lambda_+ \delta_\Sigma^{-1} \|\upsilon(x)-\tilde{\upsilon}(x)\|^2\\
&\leq \frac{10}{7} L_-^{-1} \lambda_-^{-2} \lambda_+ \delta_\Sigma^{-1}
p \gamma_\kappa L_- \lambda_-^2 \lambda_+^{-1} \delta_\Sigma^2\\
&\leq \frac{10}{7} p \gamma_\kappa \delta_\Sigma
\end{align*}
Since $\displaystyle \sqrt{\frac{|\tilde{\Sigma}|}{|\Sigma|}}=\left(\frac{\tilde{L}}{L}\right)^{\frac{p}{2}}\leq \left(1+\frac{2}{25}\delta_\Sigma\right)^{p/2}$,
\begin{align*}
\frac{\Phi_{\upsilon(x),\Sigma}(y)}{(1+\kappa \delta_\Sigma)^{p}
    \Phi_{\tilde{\upsilon}(x),(1+\delta_\Sigma) \tilde{\Sigma}}(y)}&
\leq \frac{(1+\delta_\Sigma)^{p/2}(1+\frac{2}{25}\delta_\Sigma)^{p/2}}{(1+\kappa \delta_\Sigma)^{p}} \exp\left(\frac{5 \gamma_\kappa}{7}p\delta_\Sigma\right).
\end{align*}
It suffices that \[
\frac{5 \gamma_\kappa}{7}\delta_\Sigma \leq
\ln\left(\frac{1+\kappa \delta_\Sigma}{\sqrt{1+\delta_\Sigma}\sqrt{1+\frac{2}{25}\delta_\Sigma}} \right)
\]
Now let 
\begin{align*}
f(\delta_\Sigma)&=\ln(1+\kappa \delta_\Sigma)-\frac{1}{2}\ln(1+\delta_\Sigma)-\frac{1}{2}\ln\left(1+\frac{2}{25}\delta_\Sigma\right)\\
f'(\delta_\Sigma)&=\frac{\kappa}{1+\kappa \delta_\Sigma}-\frac{1}{2(1+\delta_\Sigma)}-\frac{1}{25\left(1+\frac{2}{25}\delta_\Sigma\right)}=\frac{(27k-4)\delta_\Sigma+50k-27}{2(1+\kappa\delta_\Sigma)(1+\delta_\Sigma)(25+2\delta_\Sigma)}
\end{align*}
Since $\kappa>\frac{17}{29}$,
\begin{align*}
f'(\delta_\Sigma)>\frac{k-\frac{27}{50}}{(1+\kappa\delta_\Sigma)(1+\delta_\Sigma)\left(1+\frac{2}{25}\delta_\Sigma\right)}
\end{align*}
Finally, since $f(0)=0$ and $\delta_\Sigma\leq \frac{2}{5}$, one deduces
\begin{align*}
f(\delta_\Sigma)&> \frac{k-\frac{27}{50}}{(1+\kappa\delta_\Sigma)(1+\delta_\Sigma)\left(1+\frac{2}{25}\delta_\Sigma\right)} \delta_\Sigma\\
&\geq \frac{k-\frac{27}{50}}{\left(1+\frac{2}{5}\kappa\right)\left(1+\frac{2}{5}\right)\left(1+\frac{2}{25}\frac{2}{5}\right)} \delta_\Sigma =\frac{5}{7}\frac{125(k-\frac{27}{50})}{129\left(1+\frac{2}{5}\kappa\right)} \delta_\Sigma\\
&\geq \frac{5}{7}\gamma_\kappa \delta_\Sigma
\end{align*}
So $\Phi_{\upsilon,\Sigma}\leq t^+$.
$\frac{t^-}{\Phi_{\upsilon,\Sigma}}$ is handled the same way.
\begin{align*}
&\frac{(1+\kappa \delta_\Sigma)^{-p}\Phi_{\tilde{\upsilon}(x),(1+\delta_\Sigma)^{-1}\tilde{\Sigma}}(y)}{\Phi_{\upsilon(x),\Sigma}(y)}\\
&\leq (1+\kappa \delta_\Sigma)^{-p} \left(\sqrt{\frac{|\Sigma|}{|(1+\delta_\Sigma)^{-1}\tilde{\Sigma}|}}   \exp\left(\frac{1}{2}(\upsilon(x)-\tilde{\upsilon}(x))'\left(\Sigma-(1+\delta_\Sigma)^{-1}\tilde{\Sigma} \right)^{-1}(\upsilon(x)-\tilde{\upsilon}(x)) \right) \right)\\
&\leq \frac{(1+\delta_\Sigma)^{p/2}}{(1+\kappa \delta_\Sigma)^p}
\exp\left(\frac{1}{2}(\upsilon(x)-\tilde{\upsilon}(x))'\left(\Sigma-(1+\delta_\Sigma)^{-1}\tilde{\Sigma} \right)^{-1}
(\upsilon(x)-\tilde{\upsilon}(x))\right)\\
\end{align*}
Now
\begin{align*}
\left(\Sigma-(1+\delta_\Sigma)^{-1}\tilde{\Sigma} \right)^{-1}
&=\left(\Sigma \left((1+\delta_\Sigma)\tilde{\Sigma}^{-1}-\Sigma^{-1}\right)
(1+\delta_\Sigma)^{-1}\tilde{\Sigma}\right)^{-1}\\
&=(1+\delta_\Sigma)\tilde{\Sigma}^{-1} \left((1+\delta_\Sigma)\tilde{\Sigma}^{-1}-\Sigma^{-1}\right)^{-1}\Sigma^{-1}
\end{align*}
and
\begin{align*}
(\upsilon(x)-\tilde{\upsilon}(x))'\left(\Sigma-(1+\delta_\Sigma)^{-1}\tilde{\Sigma} \right)^{-1}(\upsilon(x)-\tilde{\upsilon}(x))
&\leq (1+\delta_\Sigma)\tilde{L}^{-1}\lambda_-^{-1}
2\tilde{L}\lambda_+ \delta_\Sigma^{-1}
L_-^{-1}\lambda_-^{-1}
p \gamma_\kappa L_- \lambda_-^2 \lambda_+^{-1} \delta_\Sigma^2\\
&\leq 2p\gamma_\kappa(1+\delta_\Sigma)\delta_\Sigma
\end{align*}
We only need to prove that
\begin{align*}
\gamma_\kappa(1+\delta_\Sigma)\delta_\Sigma\leq
\ln\left(\frac{1+\kappa \delta_\Sigma}{\sqrt{1+\delta_\Sigma}}\right)
\end{align*}
Let
\begin{align*}
g(\delta_\Sigma)&=\ln\left(\frac{1+\kappa \delta_\Sigma}{\sqrt{1+\delta_\Sigma}}\right)\\
g'(\delta_\Sigma)&=\frac{\kappa}{1+\kappa \delta_\Sigma}-\frac{1}{2(1+\delta_\Sigma)}
=\frac{\kappa \delta_\Sigma +2\kappa-1}{2(1+\delta_\Sigma)(1+\kappa \delta_\Sigma)}
\end{align*}
Provided that $\kappa \geq \frac{1}{2}$ and $\delta_\Sigma\leq \frac{2}{5}$,
\[
g'(\delta_\Sigma)> \frac{2\kappa-1}{2(1+\frac{2}{5})(1+\frac{2}{5}\kappa)}
.\]
Finally, since $g(0)=0$,
\[
g(\delta_\Sigma)> \frac{2\kappa-1}{2(1+\frac{2}{5})(1+\frac{2}{5}\kappa)}\delta_\Sigma
=\frac{5(2\kappa-1)}{14(1+\frac{2\kappa}{5})}\delta_\Sigma
\geq \frac{7}{5}\gamma_\kappa \delta_\Sigma
\geq \left(1+\delta_\Sigma\right)\gamma_\kappa \delta_\Sigma
.\]
One deduces $(1+\kappa \delta_\Sigma)^{-p}\Phi_{\tilde{\upsilon}(x),(1+\delta_\Sigma)^{-1}\tilde{\Sigma}}(y)\leq\Phi_{\upsilon(x),\Sigma}(y)$.
\end{proof}

\subsection{Proof of inequalities used for bracketing entropy's
  decomposition}

For sake of completeness, we prove here the inequalities of Lemma 11
and 12 of \cite{ConditionalDensity} used in the proof of Lemma~\ref{decomp lemma}. 
\begin{proof}[Proof of Lemma 11]

For all $x$ in $\mathcal{X}$,
\begin{align*}
&d_{y,k}^2(\pi^-(x) \psi^-(x,.), \pi^+(x) \psi^+(x,.))\\
=&\int \sum_{k=1}^K \left(\sqrt{\pi_k^+(x)} \left(\sqrt{\psi_k^+(x,y)}-\sqrt{\psi_k^-(x,y)} \right)\right.\\
&\qquad \qquad \qquad \qquad \qquad \qquad \qquad \left.+\sqrt{\psi_k^-(x,y)} \left(\sqrt{\pi_k^+(x)}-\sqrt{\pi_k^-(x)} \right) \right)^2 dy\\
=&\int \sum_{k=1}^K \pi_k^+(x) \left(\sqrt{\psi_k^+(x,y)}-\sqrt{\psi_k^-(x,y)} \right)^2 dy\\
+&\int \sum_{k=1}^K \psi_k^-(x,y) \left(\sqrt{\pi_k^+(x)}-\sqrt{\pi_k^-(x)} \right)^2 dy \\
+&2 \sum_{k=1}^K \sqrt{\pi_k^+(x)} \left(\sqrt{\pi_k^+(x)}-\sqrt{\pi_k^-(x)} \right) \int \sqrt{\psi_k^-(x,y)} \left(\sqrt{\psi_k^+(x,y)}-\sqrt{\psi_k^-(x,y)} \right) dy\\
\leq & \left(\sum_{k=1}^K \pi_k^+(x) \right) \max_k d_y^2(\psi_k^+(x,.),\psi_k^-(x,.)) +d_k^2(\pi^+(x),\pi^-(x)) \max_k \int \psi_k^-(x,y) dy\\
+& 2\sum_{k=1}^K \sqrt{\pi_k^+(x)} \left(\sqrt{\pi_k^+(x)}-\sqrt{\pi_k^-(x)} \right) d_y(\psi_k^+(x,.),\psi_k^-(x,.))\sqrt{\int \psi_k^-(x,y) dy}\\
\leq & \left(\sum_{k=1}^K \pi_k^+(x) \right) \max_k d_y^2(\psi_k^+(x,.),\psi_k^-(x,.)) +d_k^2(\pi^+(x),\pi^-(x)) \max_k \int \psi_k^-(x,y) dy\\
+& 2 \max_k \sqrt{\int \psi_k^-(x,y) dy} \max_k d_y(\psi_k^+(x,.),\psi_k^-(x,.)) \left(\sum_{k=1}^K \pi_k^+(x) \right)^{1/2} d_k(\pi^+(x),\pi^-(x))\\
\leq & \left( \max_k d_y(\psi_k^+(x,.),\psi_k^-(x,.)) \sqrt{\sum_{k=1}^K \pi_k^+(x)}\right. \\
&\qquad \qquad \qquad \qquad \qquad \qquad \qquad \left.+d_k(\pi^+(x),\pi^-(x)) \max_k \sqrt{\int \psi_k^-(x,y) dy} \right)^2
\end{align*}
\end{proof}

\begin{proof}[Proof of Lemma 12] 

For all $x$ in $\mathcal{X}$,
\begin{align*}
&d_y^2 \left(\sum_{k=1}^K \pi_k^-(x) \psi_k^-(x,.),\sum_{k=1}^K \pi_k^+(x) \psi_k^+(x,.)\right)
= \int \sum_{k=1}^K \pi_k^+(x) \psi_k^+(x,y) dy \\
&+\int \sum_{k=1}^K \pi_k^-(x) \psi_k^-(x,y) dy 
-2\int \sqrt{\sum_{k=1}^K \pi_k^+(x) \psi_k^+(x,y)} \sqrt{\sum_{k=1}^K \pi_k^-(x) \psi_k^-(x,y)} dy\\
 &\leq \int \sum_{k=1}^K \pi_k^+(x) \psi_k^+(x,y) dy +\int \sum_{k=1}^K \pi_k^-(x) \psi_k^-(x,y) dy\\
&\qquad \qquad \qquad \qquad \qquad \qquad -2\int \sum_{k=1}^K \sqrt{\pi_k^+(x) \psi_k^+(x,y)} \sqrt{\pi_k^-(x) \psi_k^-(x,y)} dy\\
&\leq d_{y,k}^2(\pi^-(x) \psi^-(x,.), \pi^+(x) \psi^+(x,.))
\end{align*}
\end{proof}

\subsection{Proof of lemmas used for Gaussian's bracketing entropy}

\subsubsection{Proof of Lemma~\ref{lemma16}}

\begin{proof}
\begin{align*}
\delta_\Sigma &\leq
\frac{1}{5\sqrt{\kappa^2 \cosh(\frac{2\kappa}{5}) +\frac{1}{2}}}\frac{\delta}{p}
\leq \frac{1}{5\sqrt{\kappa^2 +\frac{1}{2}}}\frac{\delta}{p}
\leq \frac{1}{5\sqrt{\left(\frac{1}{2}\right)^2 +\frac{1}{2}}}\frac{\delta}{p}
\leq \frac{2\sqrt{2}}{5\sqrt{3}p}
\leq \frac{2}{5p}
\end{align*}
\end{proof}

\subsubsection{Proof of Lemma~\ref{lemma17}}


\begin{proof}
  \begin{align*}
    2-2\, 2^{d/2}\, \left( (1+\delta_\Sigma)+(1+\delta_\Sigma)^{-1}\right)^{-d/2} & =
    2\left( 1 - \left(
        \frac{e^{\ln(1+\delta_\Sigma)}+e^{-\ln(1+\delta_\Sigma)}}{2}\right)^{-d/2}
      \right)\\
& = 2 \left( 1 - \left( \cosh\left(\ln (1+\delta_\Sigma)\right)\right)^{-d/2}
  \right)\\
&   = 2 f\left(\ln(1+\delta_\Sigma)\right)\\
  \end{align*}
where $f(x)=1-\cosh(x)^{-d/2}$. Studying this function yields
\begin{align*}
  f'(x) & = \frac{d}{2} \sinh(x) \cosh(x)^{-d/2-1}\\
  f''(x) & = \frac{d}{2} \cosh(x)^{-d/2} - \frac{d}{2} \left(
    \frac{d}{2}+1 \right) \sinh(x)^2 \cosh(x)^{-d/2-2}\\
& = \frac{d}{2} \left( 1-\left( \frac{d}{2} +1 \right) \left(
    \frac{\sinh(x)}{\cosh(x)} \right)^2 \right) \cosh(x)^{-d/2}\\
\intertext{as $\cosh(x)\geq 1$, we have thus}
f''(x) & \leq \frac{d}{2}.
\intertext{%
Now since $f(0)=0$ and $f'(0)=0$, this implies for any $x\geq0$}
  f(x) & \leq \frac{d}{2} \frac{x^2}{2} \leq \frac{d^2}{2} \frac{x^2}{2}.
\end{align*}
We deduce thus that
\begin{align*}
    2-2\, 2^{d/2}\, \left( (1+\delta_\Sigma)+(1+\delta_\Sigma)^{-1}\right)^{-d/2}
   & \leq \frac{1}{2} d^2 \left(\ln(1+\delta_\Sigma)\right)^2 \\
\intertext {and using $\ln(1+\delta_\Sigma)\leq\delta_\Sigma$}
  2-2\, 2^{d/2}\, \left( (1+\delta_\Sigma)+(1+\delta_\Sigma)^{-1}\right)^{-d/2} & \leq \frac{1}{2} d^2 \delta_\Sigma^2.
\end{align*}

Now,
\begin{align*}
  \left(1+\kappa\delta_\Sigma\right)^d +  \left(1+\kappa\delta_\Sigma\right)^{-d} -2
& = 2\left( \cosh\left(d \ln(1+\kappa\delta_\Sigma)\right)-1 \right) 
= 2 g\left(d \ln(1+\kappa\delta_\Sigma)\right)
\end{align*}
with $g(x)= \cosh(x)-1$. Studying this function yields
\begin{align*}
  g'(x)&=\sinh(x)\quad\text{and}\quad
  g''(x)  = \cosh(x)
\intertext{and thus, since $g(0)=0$ and $g'(0)=0$, for any $0\leq x \leq c$}
g(x) & \leq \cosh(c) \frac{x^2}{2}.
\end{align*}
Since $\ln(1+\kappa\delta_\Sigma) \leq \kappa\delta_\Sigma$, $d\delta_\Sigma\leq c$ implies
$d\ln(1+\kappa\delta_\Sigma) \leq \kappa c$, we obtain thus
\begin{align*}
  \left(1+\kappa\delta_\Sigma\right)^d +  \left(1+\kappa\delta_\Sigma\right)^{-d} -2
&\leq \cosh(\kappa c) d^2   \left(\ln(1+\kappa\delta_\Sigma)\right)^2
 \leq \kappa^2 \cosh(\kappa c) d^2 \delta_\Sigma^2.
\end{align*}
\end{proof}

\subsubsection{Proof of Lemma\ref{lemma18}}


\begin{proof}
By definition,
\begin{align*}
  x' \left( (1+\delta_\Sigma)\tilde{\Sigma}^{-1}-\Sigma^{-1} \right) x & =
  (1+\delta_\Sigma) \tilde{L}^{-1} \sum_{i=1}^{p} \tilde{A}_{i,i}^{-1}
  |\tilde{D}_{i}'x|^2
- L^{-1} \sum_{i=1}^{p} A_{i,i}^{-1}
  |D_{i}'x|^2\\
& = (1+\delta_\Sigma) \tilde{L}^{-1} \sum_{i=1}^{p} \tilde{A}_{i,i}^{-1}
  |\tilde{D}_{i}'x|^2
- (1+\delta_\Sigma) \tilde{L}^{-1} \sum_{i=1}^{p} \tilde{A}_{i,i}^{-1}
  |D_{i}'x|^2
 \\
& \quad + (1+\delta_\Sigma) \tilde{L}^{-1} \sum_{i=1}^{p} \tilde{A}_{i,i}^{-1}
  |D_{i}'x|^2 - (1+\delta_\Sigma) \tilde{L}^{-1} \sum_{i=1}^{p} A_{i,i}^{-1}
  |D_{i}'x|^2\\
& \quad + (1+\delta_\Sigma) \tilde{L}^{-1} \sum_{i=1}^{p} A_{i,i}^{-1}
  |D_{i}'x|^2 - L^{-1} \sum_{i=1}^{p} A_{i,i}^{-1}
  |D_{i}'x|^2
\end{align*}
Along the same lines,
\begin{align*}
  x' \left( \Sigma^{-1}-(1+\delta_\Sigma)^{-1}\tilde{\Sigma}^{-1} \right) x & =
L^{-1} \sum_{i=1}^{p} A_{i,i}^{-1}
  |D_{i}'x|^2 -   (1+\delta_\Sigma)^{-1} \tilde{L}^{-1} \sum_{i=1}^{p} \tilde{A}_{i,i}^{-1}
  |\tilde{D}_{i}'x|^2\\
& = L^{-1} \sum_{i=1}^{p} A_{i,i}^{-1}
  |D_{i}'x|^2 -   (1+\delta_\Sigma)^{-1} \tilde{L}^{-1} \sum_{i=1}^{p} A_{i,i}^{-1}
  |D_{i}'x|^2 \\
& \quad + (1+\delta_\Sigma)^{-1} \tilde{L}^{-1} \sum_{i=1}^{p} A_{i,i}^{-1}
  |D_{i}'x|^2  - (1+\delta_\Sigma)^{-1} \tilde{L}^{-1} \sum_{i=1}^{p} \tilde{A}_{i,i}^{-1}
  |D_{i}'x|^2 \\
& \quad + (1+\delta_\Sigma)^{-1} \tilde{L}^{-1} \sum_{i=1}^{p} \tilde{A}_{i,i}^{-1}
  |D_{i}'x|^2 - (1+\delta_\Sigma)^{-1} \tilde{L}^{-1} \sum_{i=1}^{p} \tilde{A}_{i,i}^{-1}
  |D_{i}'x|^2
\end{align*}

Now
\begin{align*}
 \left|  \sum_{i=1}^{p} \tilde{A}_{i,i}^{-1}
  |\tilde{D}_{i}'x|^2
-  \sum_{i=1}^{p} \tilde{A}_{i,i}^{-1}
  |D_{i}'x|^2 \right| & \leq  \sum_{i=1}^{p} \tilde{A}_{i,i}^{-1}
 \left| |\tilde{D}_{i}'x|^2 - |D_{i}'x|^2 \right|\\
& \leq  \lambda_-^{-1} \sum_{i=1}^{p}
\left| |\tilde{D}_{i}'x|^2 - |D_{i}'x|^2 \right|   \\
& \leq  \lambda_-^{-1} \sum_{i=1}^{p}
\left| |\tilde{D}_{i}'x| - |D_{i}'x| \right| \left| |\tilde{D}_{i}'x| + |D_{i}'x|
\right|\\
& \leq   \lambda_-^{-1} \left( \sum_{i=1}^{p}
\left| (\tilde{D}_{i}-D_{i})'x\right|^2 \right)^{1/2} 
\left( \sum_{i=1}^{p}
\left| (\tilde{D}_{i}+D_{i})'x\right|^2 \right)^{1/2} \\
& \leq \lambda_-^{-1} \delta_D \|x\| 2 \|x\| = \lambda_-^{-1}
2 \delta_D \|x\|^2.
\end{align*}
Furthermore,
\begin{align*}
\left|  \sum_{i=1}^{p} \tilde{A}_{i,i}^{-1}
  |D_{i}'x|^2 - \sum_{i=1}^{p} A_{i,i}^{-1}
  |D_{i}'x|^2 \right| & \leq \sum_{i=1}^{p} \left|
\tilde{A}_{i,i}^{-1} - A_{i,i}^{-1} \right|
  |D_{i}'x|^2 \\
& \leq \delta_A \lambda_-^{-1}\sum_{i=1}^{p} |D_{i}'x|^2 = 
\delta_A \lambda_-^{-1} \|x\|^2.
\end{align*}

We notice then that
\begin{align*}
(1+\delta_\Sigma) \tilde{L}^{-1} \sum_{i=1}^{p} A_{i,i}^{-1}
  |D_{i}'x|^2 - L^{-1} \sum_{i=1}^{p} A_{i,i}^{-1}
  |D_{i}'x|^2 & = \left( (1+\delta_\Sigma) \tilde{L}^{-1} -L^{-1} \right) \sum_{i=1}^{p} A_{i,i}^{-1}
  |D_{i}'x|^2\\
& \geq (\delta_\Sigma-\delta_L) \tilde{L}^{-1} \lambda_+^{-1} \|x\|^2
\end{align*}
while
\begin{align*}
  L^{-1} \sum_{i=1}^{p} A_{i,i}^{-1}
  |D_{i}'x|^2 -   (1+\delta_\Sigma)^{-1} \tilde{L}^{-1} \sum_{i=1}^{p} A_{i,i}^{-1}
  |D_{i}'x|^2 & = \left( L^{-1} - (1+\delta_\Sigma)^{-1} \tilde{L}^{-1} \right)
  \sum_{i=1}^{p} A_{i,i}^{-1}
  |D_{i}'x|^2\\
& \geq \left( 1-(1+\delta_\Sigma)^{-1} \right) \tilde{L}^{-1}
\lambda_+^{-1} \|x\|^2\\
& \geq \frac{\delta_\Sigma}{1+\delta_\Sigma} \lambda_+^{-1} \tilde{L}^{-1} \|x\|^2
\end{align*}

We deduce thus that
\begin{align*}
  x' \left( (1+\delta_\Sigma)\tilde{\Sigma}^{-1}-\Sigma^{-1} \right) x & \geq
  (\delta_\Sigma-\delta_L) \tilde{L}^{-1} \lambda_+^{-1} \|x\|^2 - (1+\delta_\Sigma)
\tilde{L}^{-1} \lambda_-^{-1}
  \left( 
2 \delta_D  + 2 \delta_A \right) \|x\|^2\\
& \geq  \tilde{L}^{-1} \left( (\delta_\Sigma-\delta_L) \lambda_+^{-1} -
  (1+\delta_\Sigma)\lambda_-^{-1} \left( 
2 \delta_D  + \delta_A \right) \right) \|x\|^2 
\intertext{and}
  x' \left( \Sigma^{-1}-(1+\delta_\Sigma)^{-1}\tilde{\Sigma}^{-1} \right) x & \geq
\frac{\delta_\Sigma}{1+\delta_\Sigma} \tilde{L}^{-1}\lambda_+^{-1} \|x\|^2
- (1+\delta_\Sigma)^{-1} \tilde{L}^{-1} \lambda_-^{-1} \left( 
2 \delta_D  + \delta_A \right) \|x\|^2\\
& \geq
\frac{\tilde{L}^{-1}}{1+\delta_\Sigma} \left(
  \delta_\Sigma \lambda_+^{-1} - \lambda_-^{-1} \left( 
2 \delta_D  + \delta_A \right) \right) \|x\|^2
\end{align*}
\end{proof}
}{}

\bibliographystyle{plainnat}
\bibliography{BiblioGaussian}

\end{document}